\documentclass[a4paper,10pt]{article}
\usepackage{amsthm,amsmath,amssymb,amsfonts,bbm,geometry,epsfig,listings,hyperref,enumerate,comment,nicefrac,verbatim,multirow,titlesec}
\usepackage{algorithm}
\usepackage{algpseudocode}
\usepackage{tabularx}

\makeatletter
\newcommand{\multiline}[1]{%
  \begin{tabularx}{\dimexpr\linewidth-\ALG@thistlm}[t]{@{}X@{}}
    #1
  \end{tabularx}
}
\makeatother

\newtheorem{lemma}{Lemma}[section]

\newtheorem{theorem}[lemma]{Theorem}

\newtheorem{corollary}[lemma]{Corollary}

\newcommand{\1}{\ensuremath{\mathbbm{1}}}

\providecommand{\N}{{\ensuremath{\mathbbm{N}}}}
\providecommand{\Z}{{\ensuremath{\mathbbm{Z}}}}
\providecommand{\R}{{\ensuremath{\mathbbm{R}}}}

\providecommand{\E}{{\ensuremath{\mathbb{E}}}}
\renewcommand{\P}{{\ensuremath{\mathbb{P}}}}
\newcommand{\eps}{{\ensuremath{\varepsilon}}}

\newcommand{\Var}{{\ensuremath{\operatorname{Var}}}}
\newcommand{\Lip}{{\ensuremath{\operatorname{Lip}}}}
\newcommand{\funcF}{F}
\newcommand{\LipConst}{L}

\title{Multilevel Picard iterations for solving smooth\\ semilinear parabolic heat equations}

\author{Weinan E$^{1} $, Martin Hutzenthaler$^{2}$, Arnulf Jentzen$^{3}$ \& Thomas Kruse$^{4}$
\bigskip
\\
\small{$^1$ Department of Mathematics and 
Program in Applied and Computational Mathematics,}
\\
\small{Princeton University, Princeton, NJ 08544-1000, USA,
e-mail: weinan@math.princeton.edu}
\smallskip
\\
\small{$^2$ Faculty of Mathematics, University of Duisburg-Essen,}
\\
\small{45117 Essen, Germany, e-mail: martin.hutzenthaler@uni-due.de}
\smallskip
\\
\small{$^3$ Seminar f\"ur Angewandte Mathematik,
ETH Zurich,}
\\
\small{8092 Z\"urich, Switzerland, e-mail: arnulf.jentzen@sam.math.ethz.ch}
\smallskip
\\
\small{$^4$ Faculty of Mathematics, University of Duisburg-Essen,}
\\
\small{45117 Essen, Germany, e-mail: thomas.kruse@uni-due.de}
}

\begin{document}

\maketitle
\makeatletter
\let\@makefnmark\relax
\let\@thefnmark\relax
\@footnotetext{\emph{AMS 2010 subject classification:} 65M75}
\@footnotetext{\emph{Key words and phrases:}
curse of dimensionality, high-dimensional PDEs, high-dimensional semilinear BSDEs, multilevel 
Picard iteration, multilevel Monte Carlo method
  }
\makeatother

\begin{abstract}
We introduce a new family of numerical algorithms for approximating solutions of general high-dimensional semilinear parabolic partial differential equations at single space-time points.
 The algorithm is obtained through a delicate combination of 
the Feynman-Kac and the Bismut-Elworthy-Li formulas,
 and  an approximate decomposition of the Picard fixed-point
iteration with multilevel accuracy.
The algorithm has been tested on a variety of semilinear partial differential equations
that arise in physics and finance, with very satisfactory results.
Analytical tools needed for the analysis of such algorithms, including
a semilinear Feynman-Kac formula, a new class of semi-norms and their
recursive inequalities, are also introduced.
They allow us to prove
for semilinear heat equations 
with gradient-independent nonlinearity
that the computational 
complexity of the proposed algorithm  is bounded by
$O(d\,\eps^{-(4+\delta)})$ for any $\delta \in (0,\infty)$ under suitable assumptions, where
$d\in \N$ is the dimensionality of the problem and 
$\eps\in(0,\infty)$ is the prescribed accuracy.

\end{abstract}


\section{Introduction and main results}
High-dimensional partial differential equations (PDEs) arise naturally in many important areas including
quantum mechanics, statistical physics, financial engineering, economics, etc. Yet developing efficient and practical algorithms for these 
high-dimensional PDEs has been a long-standing problem and indeed one of the most challenging tasks in mathematics.
The difficulty lies in the ``curse of dimensionality'' \cite{Bellman}, i.e., the  complexity of the problem goes up 
exponentially as a function of dimension, which is a well-known obstacle that is also  at the heart of many other  important
subjects such as high-dimensional statistics and the modeling of many-body systems.

For  \emph{linear} parabolic PDEs, the  Feynman-Kac formula
establishes an explicit representation of the solution of the PDE 
as the expectation of the solution of an appropriate 
stochastic differential equation (SDE).
Monte Carlo methods 
together with suitable discretizations of the SDE (see, e.g.,
\cite{m55,kp92,hjk12,HutzenthalerJentzen2014})
then allow to approximate the solution at any single point in space-time 
with a computational complexity  that grows  
as $O(d\eps^{-(2+\delta)})$ for any $\delta > 0$ where 
$d$ is the dimensionality of the problem and $\eps$ is the accuracy required
(cf., e.g., \cite{GrahamTalay2013Stochastic_simulation_and_Monte_Carlo_methods,g08b,h98,heinrich01}).

In the seminal papers~\cite{PardouxPeng1990,Peng1991,PardouxPeng1992},
Pardoux \& Peng 
established a generalized \emph{nonlinear Feynman-Kac formula} that gives
 an explicit representation of the solutions of  a semilinear parabolic PDE
through the solution of an appropriate 
backward stochastic differential equation (BSDE).
Solving the BSDEs numerically, however, requires in general suitable discretizations of nested
conditional expectations 
(see, e.g., \cite{BouchardTouzi2004,Zhang2004})
and the straightforward Monte Carlo method applied 
to these nested conditional expectations results in an algorithm
with a computational complexity that grows  
polynomially in $d$ but exponentially  in $\eps^{-1}$.
Other discretization methods for the 
nested conditional expectations proposed in the literature include
the quantization tree method (see~\cite{BallyPages2003a}),
the regression method based on Malliavin calculus or based on kernel estimation (see~\cite{BouchardTouzi2004}),
the projection on function spaces method (see~\cite{GobetLemorWarin2005}),
the cubature on Wiener space method (see~\cite{CrisanManolarakis2012}),
and the Wiener chaos decomposition method (see~\cite{BriandLabart2014}).
None of these algorithms meets  the requirement that the computational complexity grows at most polynomially
both in $d$ and $\eps^{-1}$
(see \cite[Subsections 6.1--6.6]{EHutzenthalerJentzenKruse2017}
for a detailed discussion of these approximation methods).

Another probabilistic representation for the
solutions of some semilinear parabolic PDEs with polynomial nonlinearity
has been established in Skorohod \cite{Skorohod1964} by means 
of \emph{branching diffusion processes}.
Recently this classical representation
has been extended to more general analytic nonlinearities
\cite{Henry-Labordere2012, Henry-LabordereTanTouzi2014, Henry-LabordereEtAl2016}.
This probabilistic representation has been successfully used 
to obtain a Monte Carlo approximation method for semilinear parabolic PDEs
with a computational complexity that grows polynomially both in $d$ and $\eps^{-1}$.
However, not only is this method only applicable to a special class of PDEs,  it also
requires the terminal/initial condition to be quite small 
(see  \cite[Subsection 6.7]{EHutzenthalerJentzenKruse2017}
for a detailed discussion).

In this paper
we propose a new family of numerical algorithms for approximating solutions of general high-dimensional
semilinear parabolic PDEs (and BSDEs) at single space-time points; 
see~\eqref{eq:def_scheme} below for the definition of our approximations.
For semilinear heat equations with gradient-independent nonlinearities
we prove that the computational complexity
(see Corollary \ref{c:rate4} below for the precise meaning hereof)
of our proposed algorithm is $O(d\,\eps^{-(4+\delta)})$
for any $\delta > 0$ under suitable assumptions including the strong smoothness assumption
that the constant in~\eqref{eq:constant_rate4} below is finite; see Corollary \ref{c:rate4} below for details.
Under the assumptions of Corollary \ref{c:rate4}, to the best of our knowledge, no implementable
approximation method was known in the literature to overcome the curse of dimensionality.
The analysis of more general coefficient functions and nonlinearities is deferred to future publications.
The algorithm, which we will call ``multilevel Picard iteration'',
 is a delicate combination of the Feynman-Kac and Bismut-Elworthy-Li formulas,
 and a decomposition of the
 Picard iteration with multilevels of accuracy.
 The efficiency and accuracy of the proposed algorithm has been tested on a variety of
 semilinear parabolic PDEs that arise in physics and finance.
 These details are presented in \cite{EHutzenthalerJentzenKruse2017}. 
To get a feeling about the performance of the algorithm:
 To evaluate $u(1, 0)$ for the solution of 
\begin{equation}\label{eq:allen_cahn_intro}
\partial_t  u = \tfrac{1}{ 2 }\Delta u  + u-u^3
\end{equation}
with $d=100, \eps=0.01, u( 0, x) =( 1 + \max\{ | x_1 |^2, ..., | x_{100} |^2)^{-1} $ requires
$10$ seconds of runtime on a 2.8 GHz Intel i7 processor with 16 GB RAM. 
     
We also introduce the tools needed to analyze these high-dimensional algorithms.
Some of these tools are quite  non-standard (e.g.  the semi-norms~\eqref{eq:seminorms}
and the recursive inequality~\eqref{eq:global.estimate} involving different semi-norms).
Using these tools, we are able to establish rigorously the bounds for the computational complexity 
mentioned above.

\subsection{Notation}
Since the proposed algorithm relies heavily on the Feynman-Kac formula, we will adopt the notations and 
conventions in stochastic analysis. In addition, we frequently use the following notation.
We denote by
$
  \left\| \cdot \right\| \colon
  \left(
    \cup_{ n \in \N }
    \R^n
  \right)
  \to 
  [0,\infty)
$
and
$
  \langle \cdot, \cdot \rangle \colon
  \left(
    \cup_{ n \in \N }
    (
    \R^n \times \R^n
    )
  \right)
  \to
  [0,\infty)
$
the functions that satisfy 
for all $ n \in \N $, $ v = ( v_1, \dots, v_n ) $, $ w = ( w_1, \dots, w_n ) \in \R^n $
that
$
  \left\| v \right\|
  =
  \big[ 
    \sum_{i=1}^n\left| v_i \right|^2
  \big]^{ 1 / 2 }
$
and
$
  \langle
    v, w
  \rangle
    =
   \sum_{i=1}^n v_i w_i
$.
For every
topological space $(E,\mathcal E)$ we denote by $\mathcal{B}(E)$
the Borel-sigma-algebra on  $(E,\mathcal E)$.
For all measurable spaces $(A,\mathcal{A})$ and $(B,\mathcal{B})$
we denote by $\mathcal{M}(\mathcal A,\mathcal B)$ the set of $\mathcal{A}$/$\mathcal{B}$-measurable
functions from $A$ to $B$.
For all metric spaces $(E,d_E)$ and $(F,d_F)$ we denote by
$\Lip(E,F)$ the set of all globally Lipschitz continuous functions from $E$ to $F$.
For every $d\in\N$
we denote by $\R^{d\times d}_{ \operatorname{Inv} }$ the set of invertible matrices in $\R^{d\times d}$.
For every $d\in\N$ and every $A\in\R^{d\times d}$ we denote by $A^{*}\in\R^{d\times d}$ 
the transpose of $A$.
For every $d\in\N$ and every $x=(x_1,\ldots,x_d)\in\R^d$ we denote by $\operatorname{diag}(x)\in\R^{d\times d}$
 the diagonal matrix with diagonal entries $x_1,\ldots,x_d$.
For every $T\in(0,\infty)$ we denote by $\mathcal{Q}_T$ the set given by
$\mathcal{Q}_T=\{w\colon[0,T]\to\R\colon w^{-1}(\R\backslash\{0\})\text{ is a finite set}\}$.
We denote by $\lfloor\cdot\rfloor\colon\R\to\Z$
and $[\cdot]^{+}\colon\R\to[0,\infty)$
 the functions that satisfy for all $x\in\R$ that
$\lfloor x\rfloor =\max(\Z\cap(-\infty,x])$.
and $[x]^+=\max\{x,0\}$.
We denote by $\tfrac{0}{0}$, $0\cdot\infty$, and $0^0$
the real numbers given by
 $\tfrac{0}{0}=0$, $0\cdot\infty=0$, and $0^0=1$.
%
%

\section{Multilevel Picard iteration for semilinear parabolic PDEs}
\label{sec:semilinear}

\subsection{A fixed-point equation for semilinear PDEs}
\label{sec:derivation.semilinear}

Let 
$ T > 0 $, 
$ d \in \N $, 
let 
$ 
  g \colon \R^d \to \R
$, 
$
  f \colon \R \times \R^{d} \to \R  
$,
$
  u \colon [0,T] \times \R^d \to \R
$,
$
  \mu \colon [0,T] \times \R^d \to \R^d
$,
and
$ 
  \sigma = ( \sigma_1,\ldots,\sigma_d )
  \colon [0,T] \times \R^d \to \R^{ d \times d }_{ \operatorname{Inv} } 
$
be sufficiently regular functions, 
 assume 
that
$
  u(T,x) = g(x) 
$
and
\begin{equation}  
\label{eq:PDE_quasilinear_0}
 \partial_t u
  + 
  f( 
    u    , 
    \sigma^{*}\nabla u
  )
  +
  \langle 
    \mu
    ,
    \nabla u  
  \rangle
  +
  \tfrac{ 1 }{ 2 }
  \operatorname{Trace}(
    \sigma \sigma ^*
    \operatorname{Hess} u
  ) 
  = 0
  ,   
\end{equation}
for
$ t \in [0,T) $, $ x \in \R^d $,
let
$
  ( 
    \Omega, \mathcal{F}, \P, ( \mathbb{F}_t )_{ t \in [0,T] } 
  )
$
be a stochastic basis 
(cf., e.g., \cite[Appendix~E]{PrevotRoeckner2007}), let
$
  W = 
$
$
  ( W^{ 1 }, \dots, 
$
$  
  W^{ d } ) 
  \colon 
  [0,T] \times \Omega \to \R^d
$
be a standard $ ( \mathbb{F}_t )_{ t \in [0,T] } $-Brownian motion,
and for every $ s \in [0,T] $, $ x \in \R^d $
let 
$ 
  X^{ s, x }
  \colon [s,T] \times \Omega \to \R^d
$
and
$
  D^{ s, x } 
  \colon [s,T] \times \Omega \to \R^{ d \times d }
$
be $ ( \mathbb{F}_t )_{ t \in [s,T] } $-adapted stochastic processes 
with continuous sample paths which satisfy
that for all $ t \in [s,T] $ 
it holds $ \P $-a.s.\ that
\begin{equation}  
\label{eq:def_XD}
\begin{split} 
  X^{ s, x }_t
&
  = 
  x + \int_s^t \mu(r, X^{ s, x }_r ) \, dr
  + 
  \sum_{ j = 1 }^d
  \int_s^t \sigma_j(r, X^{ s, x }_r ) \, d W^j_r,
\\
  D^{ s, x }_t 
  &
  =
  \operatorname{I}_{\R^{d\times d}} 
  +
  \int_s^t (\tfrac{\partial}{\partial x}\mu)( r,X^{ s, x }_r ) \, D^{ s, x }_r \, dr
  +
  \sum_{ j = 1 }^d
  \int_s^t (\tfrac{\partial}{\partial x}\sigma_j)( r,X^{ s, x }_r ) \, D^{ s, x }_r \, d W^{ j }_r
\end{split}     
\end{equation}
(cf., e.g., \cite[Chapter 5]{ks91}, \cite{gk96b}, or \cite{albeverio1991stochastic} for existence and uniqueness results for stochastic differential
equations of the form \eqref{eq:def_XD}).
For every $s\in[0,T]$ the processes $ D^{ s, x } $, $x\in \R^d$, are in a suitable sense the 
\emph{derivative processes} of $ X^{ s, x } $, $x\in \R^d$, with respect to $ x \in \R^d $.
%
%
%
%
%
%
Using the \emph{Feynman-Kac formula}, we have from  \eqref{eq:PDE_quasilinear_0} 
\begin{equation}\label{eq:fk_solution}
u(s,x)= \E[g(X^{s,x}_T)]+\int_s^T \E\!\left[f\!\left(u(t,X^{s,x}_t),[\sigma(t,X^{s,x}_t)]^{*}(\nabla u) (t,X^{s,x}_t)\right)\right]dt
\end{equation}
for all $(s, x) \in [0,T)\times \R^d$.
In \eqref{eq:fk_solution} the derivative of $u$ appears on the right-hand side and, therefore, \eqref{eq:fk_solution} does not provide a closed fixed point equation.
To obtain such a closed fixed point equation we now bring the \emph{Bismut-Elworthy-Li formula} into play
(see, e.g., Elworthy \& Li~\cite[Theorem~2.1]{elworthy1994formulae} or Da Prato \& Zabczyk~\cite[Theorem~2.1]{DaPratoZabczyk1997}).
This gives us
\begin{equation}
\begin{split}\label{eq:bel_solution}
[\sigma(s,x)]^{*}(\nabla u)(s,x)&= \E\!\left[g(X^{s,x}_T)\tfrac{  [\sigma(s,x)]^{*} }{ T - s } 
      \smallint\nolimits_s^T
      \big[ 
        \sigma( r, X_r^{ s, x } )^{ - 1 } 
        D_r^{ s, x } 
      \big]^{ * }
      d W_r\right]\\
      &\quad
      +\int_s^T \E\!\left[f\!\left(u(t,X^{s,x}_t),[\sigma(t,X^{s,x}_t)]^{*}(\nabla u) (t,X^{s,x}_t)\right)\tfrac{[\sigma(s,x)]^{*} }{ t - s }
      \smallint\nolimits_s^t
      \big[ 
        \sigma( r, X_r^{ s, x } )^{ - 1 } 
        D_r^{ s, x } 
      \big]^{ * } 
      \, d W_r \right]dt,
      \end{split}
\end{equation}
for all $(s, x) \in [0,T)\times \R^d$.
Now let 
$ 
  {\bf u}^{ \infty }
  \in
  \Lip( [0,T] \times \R^d, \R^{ 1 + d } )
$
be defined by 
$
  {\bf u}^{ \infty }( s, x ) 
  = 
  \big( u(s,x), [\sigma(s,x)]^{*}(\nabla u )( s, x ) \big)
$
for all $ (s,x) \in [0,T) \times \R^d $. 
Let 
$ 
  \Phi \colon 
  \Lip( [0,T] \times \R^d, \R^{ 1 + d } )
  \to
  \Lip( [0,T] \times \R^d, \R^{ 1 + d } )
$
be defined by 
\begin{equation}
\label{eq:def_Phi_semilinear}
\begin{split}
  \big( {\bf }\Phi( {\bf v} ) \big)( s, x )
& =
  \E\!\left[ 
      g( X^{ s, x }_T ) 
    \left( 
      1, 
      \tfrac{ [\sigma(s,x)]^{*} }{ T - s } 
      \smallint\nolimits_s^T
      \big[ 
        \sigma( r, X_r^{ s, x } )^{ - 1 } 
        D_r^{ s, x } 
      \big]^{ * }
      d W_r
    \right)
  \right]
\\ 
&
\quad
  +
    \int_s^T 
    \E\!\left[
      f\!\left(
         {\bf v}\!\left(t, X_t^{ s, x } \right)
      \right) 
      \big(
      1 ,
      \tfrac{ [\sigma(s,x)]^{*} }{ t - s }
      \smallint\nolimits_s^t
      \big[ 
        \sigma( r, X_r^{ s, x } )^{ - 1 } 
        D_r^{ s, x } 
      \big]^{ * } 
      \, d W_r 
      \big)
    \right]
  dt.
\end{split}
\end{equation}
for all
$ 
  {\bf v} \in 
  \Lip( [0,T] \times \R^d, \R^{ 1 + d } )
$, $ (s,x) \in [0,T) \times \R^d $. Combining \eqref{eq:def_Phi_semilinear} with 
 \eqref{eq:fk_solution} and \eqref{eq:bel_solution}
gives
\begin{equation}
{\bf u}^{ \infty } = \Phi( {\bf u}^{ \infty } ).
\end{equation}
Next we define a sequence of Picard iterations associated to \eqref{eq:def_Phi_semilinear},
\begin{equation}
{\bf u}_k(s,x) = ( {\bf \Phi}( {\bf u}_{ k - 1 } ))(s,x)
\end{equation}
for all $ k \in \N $, $s\in [0,T)$, $x\in \R^d$.
This sequence of Picard iterations has already been studied in the literature;
see, e.g.\ Thereom 7.3.4 in~\cite{YongZhou1999} or \cite{BenderDenk2007}.
Under suitable assumptions, e.g., Thereom 7.3.4 in~\cite{YongZhou1999}
ensures that for all $s\in [0,T)$, $x\in \R^d$ it holds that
$ \lim_{ k \to \infty } {\bf u}_k(s,x) = {\bf u}^{ \infty }(s,x) $. Observe that for all
$ k \in \N $, $s\in [0,T)$, $x\in \R^d$ it holds that
\begin{equation}
\label{eq:multilevel_fixed_point2}
\begin{split}
  {\bf u}_k (s,x)
  &=
  {\bf u}_1(s,x)
  +
  \sum_{ l = 1 }^{ k - 1 }
  \left[
    {\bf u}_{ l + 1 }(s,x)
    -
    {\bf u}_l(s,x)
  \right]
=
  ({\bf \Phi}( {\bf u}_0 ))(s,x)
  +
  \sum_{ l = 1 }^{ k - 1 }
  \Big[
    ({\bf \Phi}( {\bf u}_l ))(s,x)
    -
   ( {\bf \Phi}( {\bf u}_{ l - 1 } ))(s,x)
  \Big]\\
  &=
  \E\!\left[ 
      g( X^{ s, x }_T )
    \left( 
      1, 
      \tfrac{ [\sigma(s,x)]^{*} }{ T - s } 
      \smallint\nolimits_s^T
      \big[ 
        \sigma( r, X_r^{ s, x } )^{ - 1 } 
        D_r^{ s, x } 
      \big]^{ * }
      d W_r
    \right)
  \right]\\
  &\quad
  +\sum_{l=0}^{k-1}\int_s^T 
    \E\!\left[
    \left(
      f\!\left(
         {\bf u}_{l}\!\left(t, X_t^{ s, x } \right)
      \right) 
      -
      \1_{\N}(l)
      f\!\left(
         {\bf u}_{l-1}\!\left(t, X_t^{ s, x } \right)
      \right) 
      \right)
      \big(
      1 ,
      \tfrac{ [\sigma(s,x)]^{*} }{ t - s }
      \smallint\nolimits_s^t
      \big[ 
        \sigma( r, X_r^{ s, x } )^{ - 1 } 
        D_r^{ s, x } 
      \big]^{ * } 
      \, d W_r 
      \big)
    \right]
  dt
  .
\end{split}
\end{equation}
Next we incorporate a zero expectation term to slightly reduce the variance 
when approximating the expectation involving $g$ by Monte Carlo approximations.
More precisely, for all
$ k \in \N $, $s\in [0,T)$, $x\in \R^d$ it holds that
\begin{equation}
\label{eq:multilevel_fixed_point3}
\begin{split}
  {\bf u}_k (s,x)
  &=
  (g(x),0)+
  \E\!\left[ 
  \left(
      g( X^{ s, x }_T )-g(x)
    \right)
    \left( 
      1, 
      \tfrac{ [\sigma(s,x)]^{*} }{ T - s } 
      \smallint\nolimits_s^T
      \big[ 
        \sigma( r, X_r^{ s, x } )^{ - 1 } 
        D_r^{ s, x } 
      \big]^{ * }
      d W_r
    \right)
  \right]\\
  &\quad
  +\sum_{l=0}^{k-1}\int_s^T 
    \E\!\left[
    \left(
      f\!\left(
         {\bf u}_{l}\!\left(t, X_t^{ s, x } \right)
      \right) 
      -
      \1_{\N}(l)
      f\!\left(
         {\bf u}_{l-1}\!\left(t, X_t^{ s, x } \right)
      \right) 
      \right)
      \big(
      1 ,
      \tfrac{ [\sigma(s,x)]^{*} }{ t - s }
      \smallint\nolimits_s^t
      \big[ 
        \sigma( r, X_r^{ s, x } )^{ - 1 } 
        D_r^{ s, x } 
      \big]^{ * } 
      \, d W_r 
      \big)
    \right]
  dt
  .
\end{split}
\end{equation}
In this telescope expansion, we will apply a fundamental idea of 
Heinrich~\cite{h98,heinrich01}
and Giles~\cite{g08a} 
(control variates were also used, e.g., in~\cite{k05,GobetLabart2010})
and 
approximate
the continuous quantities (expectation and time integral) by discrete ones (Monte Carlo averages and quadrature formulas respectively)
with different degrees of accuracy at different levels of the Picard iteration.
Since for large $l\in \N$ the difference between $ {\bf u}_{l}$ and ${\bf u}_{l-1}$ is small, say $\rho^{-l}$, it suffices to approximate the expectation and the time integral with lower accuracy, say $\rho^{-(k-l)}$, at level $l\in \{0,\ldots,k-1\}$ for the $k$-th approximation.
More precisely,
we denote by
$ 
  ( q^{ k,\rho }_s )_{ k \in \N_0, \rho \in (0,\infty), s \in [0,T) } 
  \subseteq \mathcal{Q}_T 
$
a family of quadrature formulas on $C([0,T],\R)$ that we employ to 
approximate the time integrals $ \int_s^T \dots dt $, $s\in[0,T]$, appearing on the right-hand
side of \eqref{eq:multilevel_fixed_point3}.
We denote by 
$
  \Theta = \cup_{ n \in \N } \R^n
$ a set that allows to index families of independent random variables which we 
need for the Monte Carlo approximations. We denote by 
$
  ( \mathfrak{m}_{ k, \rho } )_{ k\in \N_0, \rho \in (0,\infty) }$
  and
  $
  ( m_{ k, \rho } )_{ k \in \N_0, \rho \in (0,\infty) } \subseteq \N
$
families of natural numbers that specify the number of Monte Carlo samples for 
approximating the expectations involving $g$ and $f$ on the right-hand side of \eqref{eq:multilevel_fixed_point3}.
In Section~\ref{sec:setting.full.discretization} 
we will take
$\mathfrak{m}_{k,\rho}=m_{k,\rho}=\rho^k$
for every $k\in\N_0$, $\rho\in(0,\infty)$ 
and we take $q^{k,\rho}$ as the Gau\ss-Legendre quadrature rule with
$\lfloor \rho \rfloor$ 
nodes.
Furthermore, for every $k\in\N_0$, $\rho\in(0,\infty)$, $\theta\in\Theta$, $(s,x)\in[0,T]\times\R^d$ we denote by $
  ( \mathcal{X}_{ k, \rho}^{ \theta }( s,x, t ))_{ t \in [s,T] }
$
and 
$
  ( \mathcal{I}_{ k, \rho}^{\theta }( s,x, t ) )_{ t \in (s,T] }
$
the stochastic processes that 
we employ to approximate the processes
$ 
  ( X^{ s, x }_t )_{ t \in [s,T] }
$
and 
$ 
   \big(
      1 ,
      \tfrac{ [\sigma(s,x)]^{*} }{ t - s }
      \smallint\nolimits_s^t
      \big[ 
        \sigma( r, X_r^{ s, x } )^{ - 1 } 
        D_r^{ s, x } 
      \big]^{ * } 
      \, d W_r 
      \big) _{ t \in (s,T] }
$.
More specifically, we choose for every $k\in\N_0$, $\rho\in(0,\infty)$, $\theta\in\Theta$, $(s,x)\in[0,T]\times\R^d$
the processes $
  ( \mathcal{X}_{ k, \rho}^{ \theta }( s,x, t ))_{ t \in [s,T] }
$
and 
$
  ( \mathcal{I}_{ k, \rho}^{\theta }( s,x, t ) )_{ t \in (s,T] }
$ such that for all $t\in (s,T]$,
\begin{equation}
\begin{split}
\mathcal{X}_{ k, \rho}^{ \theta }( s,x, t ) &\approx X^{s,x}_t, \\
\mathcal{I}_{ k, \rho}^{\theta }( s,x, t ) &\approx \big(
      1 ,
      \tfrac{ [\sigma(s,x)]^{*} }{ t - s }
      \smallint\nolimits_s^t
      \big[ 
        \sigma( r, X_r^{ s, x } )^{ - 1 } 
        D_r^{ s, x } 
      \big]^{ * } 
      \, d W_r 
      \big).
\end{split}
\end{equation}

\subsection{The approximation scheme}
\label{sec:algorithm_semilinear}

Let 
$ T \in (0,\infty) $, 
$ d \in \N $, 
$
  \Theta = \cup_{ n \in \N } \R^n
$,
let 
$ 
  g \colon \R^d \to \R
$, 
$
  f \colon \R^{d+1} \to \R  
$,
$
  \mu \colon [0,T] \times \R^d \to \R^d
$,
$ 
  \sigma
  \colon [0,T] \times \R^d \to \R^{ d \times d }_{ \operatorname{Inv} } 
$
be measurable functions,
let 
$ 
  ( q^{ k, \rho }_s )_{ k\in \N_0, \rho \in (0,\infty), s \in [0,T) } 
  \subseteq \mathcal{Q}_T 
$,
$
  ( \mathfrak{m}_{ k, \rho } )_{ k\in \N_0, \rho \in (0,\infty) } ,$
  $
  ( m_{ k,  \rho } )_{ k \in \N_0, \rho \in (0,\infty) } \subseteq \N
$,
let
$
  ( 
    \Omega, \mathcal{F}, \P, ( \mathbb{F}_t )_{ t \in [0,T] } 
  )
$
be a stochastic basis,
let 
$
  W^{ \theta } \colon [0,T] \times \Omega \to \R^d 
$,
$ \theta \in \Theta $,
be independent
standard $ ( \mathbb{F}_t )_{ t \in [0,T] } $-Brownian motions
with continuous sample paths,
for every 
$ l \in \Z $,
$ \rho \in (0,\infty) $,
$ \theta \in \Theta $,
$ x \in \R^d $,
$ s \in [0,T) $, $ t \in [s,T] $
let 
$
  \mathcal{X}_{ l, \rho}^{ \theta }( s,x, t )
  \colon  \Omega \to \R^d
$ and
$
  \mathcal{I}_{ l, \rho}^{ \theta }( s,x,  t)
  \colon  \Omega \to \R^{ 1+d  }
$
be functions,
and 
for every
$ \theta \in \Theta $,
$
  \rho \in (0,\infty)
$
let
$ 
  {\bf U}^{ \theta }_{ k, \rho } 
  \colon [0,T]\times\R^d \times \Omega \to \R^{ d + 1 } 
$,
$
  k \in \N_0 
$,
be 
functions
that satisfy for all $k\in\N$,
$ (s,x) \in [0,T)\times\R^d $
that
\begin{equation}\label{eq:def_scheme} 
\begin{split}
  {\bf U}^{ \theta }_{ k, \rho }( s, x )
&= 
  (g(x),0)+
  \sum_{ i = 1 }^{ \mathfrak{m}_{ k, \rho } }
  \tfrac{ 1 }{
    \mathfrak{m}_{ k, \rho } 
  }
  \,
  \big[
    g(
      \mathcal{X}_{ k, \rho }^{ (\theta, 0, -i) }(s, x, T)
    )
    -
    g(x)
  \big]
  \,
  \mathcal{I}_{ k, \rho}^{ ( \theta, 0, - i ) }( s, x, T)
\\
  &\quad +
  \sum_{ l = 0 }^{ k - 1 }
  \sum_{ i = 1 }^{
    m_{ k- l, \rho } 
  }
  \sum_{ t \in [s,T] }
  \tfrac{ 
    q^{ k- l, \rho }_s( t )
  }{
    m_{ k - l, \rho } 
  }
  \,
  \Big[
    f\Big(
      {\bf U}^{ ( \theta, l, i , t) }_{ l, \rho }\big( 
        t,
          \mathcal{X}_{ 
            k - l , \rho}^{ 
            ( \theta, l, i ) 
          }
         ( s, x, t )
      \big)
    \Big)
\\ 
& \quad
   -
   \1_{ \N }( l )
   \,
    f\Big(
      {\bf U}^{ ( \theta, -l, i, t ) }_{ [ l - 1 ]^{+} , \rho }\big( 
        t,
          \mathcal{X}_{ 
            k - l , \rho}^{
            ( \theta, l, i )} ( s, x, t )
      \big)
    \Big)
  \Big]
  \,
  \mathcal{I}_{ k - l, \rho}^{ (\theta, l, i) }( s, x, t )
  .
  \end{split}
  \end{equation}
  
Observe that the approximation scheme~\eqref{eq:def_scheme} 
employs Picard fixed-point iteration (cf., e.g., \cite{BenderDenk2007}),
multilevel/multigrid techniques (see, e.g., \cite{h98,heinrich01,g08b,cdmr09}),
discretizations of the SDE system~\eqref{eq:def_XD},
as well as quadrature approximations for the time integrals.
The numerical approximations~\eqref{eq:def_scheme} are 
full history recursive in the sense that for every $(k,\rho)\in\N\times(0,\infty)$ 
the full history 
$ {\bf U}^{ ( \cdot ) }_{ 0, \rho } $,
$ {\bf U}^{ ( \cdot ) }_{ 1, \rho } $,
$ \dots $,
$ {\bf U}^{ ( \cdot ) }_{ k - 1, \rho } $
needs to be computed recursively 
in order to compute 
$
  {\bf U}^{ ( \cdot ) }_{ k, \rho }
$.
In this sense the numerical approximations~\eqref{eq:def_scheme}
are full history recursive multilevel Picard approximations.  
Finally we remark that all multilevel Picard approximations on the right-hand side of~\eqref{eq:def_scheme} 
are independent since all Brownian motions $W^{\theta}$, $\theta\in\Theta$, are independent.
This independence is useful for the mathematical analysis and allows
an implementation with a simple recursive structure (cf.\ Subsection~\ref{sec:code}).

\subsection{Numerical simulations of high-dimensional semilinear PDEs}
\label{sec:numerics}

We applied the algorithm \eqref{eq:def_scheme} to 
approximate the solutions at single space-time points of several 
semilinear PDEs from physics and 
financial mathematics such as 
\begin{enumerate}[(i)]
\item \label{it:counterparty} a PDE arising from the recursive pricing model with default risk due to Duffie, Schroder, \& Skiadas~\cite{DuffieSchroderSkiadas1996},
\item \label{it:cva} a PDE arising from the valuation of derivative contracts with counterparty credit risk
(see, e.g., Burgard \& Kjaer~\cite{BurgardKjaer2011} and Henry-Labord\`ere~\cite{Henry-Labordere2012} for derivations of the PDE),
\item \label{it:borrowlend} a PDE arising from pricing models for financial markets with different interest rates for 
 borrowing and lending due to Bergman~\cite{Bergman1995},
\item \label{it:allencahn} a version of the Allen-Cahn equation with a double well potential, and
\item \label{it:explicit} a PDE  with an explicit solution whose three-dimensional version has been considered in Chassagneux~\cite{Chassagneux2014}.
\end{enumerate}
We took $d=100$.
All simulations are performed on a computer with a 2.8 GHz Intel i7 processor and 16 GB RAM. 
We refer to \cite{EHutzenthalerJentzenKruse2017} for 
 the simulation results, {\sc Matlab} codes and further details concerning the numerical simulations.
These results suggest that the proposed algorithm is highly efficient and quite practical for
dealing with these high-dimensional PDEs.

\section{Convergence rate for the multilevel Picard iteration}
\label{sec:convergence_rate}
In this section we establish the convergence rate 
for semilinear heat equations in the case where
the nonlinearity is independent of the gradient of the solution and
satisfies the Lipschitz-type condition~\eqref{eq:fLipschitz} below 
and when the Gau\ss-Le\-gen\-dre formula~\eqref{eq:def_gauss_leg}
(see, e.g., \cite{davis2007methods} for more details)
is used as the quadrature rule.

%
%
%
\subsection{Setting}\label{sec:setting.full.discretization}

Let 
$ T,L \in (0,\infty) $, 
$ d \in \N $, 
$ 
  g \in C^2(\R^d,\R)
$,
$
  \Theta = \cup_{ n \in \N } \R^n
$,
let
$
  ( 
    \Omega, \mathcal{F}, \P, ( \mathbb{F}_t )_{ t \in [0,T] } 
  )
$
be a stochastic basis,
let
$
  W^{ \theta } \colon [0,T] \times \Omega \to \R^d 
$,
$ \theta \in \Theta $,
be independent standard $(\mathbb{F}_t)_{t\in[0,T]}$-Brownian motions
with continuous sample paths,
let 
$
  f\colon[0,T]\times\R^d\times\R\to\R
$
  be a Borel measurable function which satisfies for all $t\in[0,T]$, $x\in\R^d$, $u_1,u_2\in\R$ that
  \begin{equation}  \begin{split}\label{eq:fLipschitz}
    |f(t,x,u_1)-f(t,x,u_2)|\leq L|u_1-u_2|,
  \end{split}     \end{equation}
let
$
  \funcF \colon 
  \mathcal{M}(\mathcal{B}([0,T]\times\R^d),\mathcal{B}(\R))
  \to
  \mathcal{M}(\mathcal{B}([0,T]\times\R^d),\mathcal{B}(\R))
$ be the function which satisfies for all $t\in[0,T]$, $x\in\R^d$, $u\in \mathcal{M}(\mathcal{B}([0,T]\times\R^d),\mathcal{B}(\R))$
  that
$
  (F(u))(t,x)=f(t,x,u(t,x))
$,
let $u^{\infty}=(u^{\infty}(r,y))_{(r,y)\in[0,T]\times\R^d}\in C^{1,2}([0,T]\times\R^d,\R)$ satisfy
for all $r\in[0,T]$, $y\in \R^d$ that
$u^{\infty}(T,y)=g(y)$
and
\begin{equation}  \begin{split}\label{eq:PDE}
  \partial_r u^{\infty}(r,y)
+\frac{1}{2}(\Delta_y u^{\infty})(r,y)
 +(\funcF(u^{\infty}))(r,y)=0,
\end{split}     \end{equation}
for every $n\in\N$ let
$(c_i^{n})_{i\in\{1,\ldots,n\}}\subseteq[-1,1]$
be the $n$ distinct roots of the Legendre polynomial $[-1,1]\ni x\mapsto \tfrac{1}{2^nn!}\tfrac{d^n}{dx^n}[(x^2-1)^n]\in\R$,
for every $n\in\N$, $a \in \R$, $b\in [a,\infty)$
let 
$q^{n,[a,b]}\colon[a,b]\to\R$ be the function which satisfies for all $t\in[a,b]$ that
\begin{equation}  \begin{split}\label{eq:def_gauss_leg}
  q^{n,[a,b]}(t)=\begin{cases}
 \int_a^b \left[\prod_{\substack{i\in\{1,\ldots,n\},\\ c_i^n\neq \frac{2t-(a+b)}{b-a}}}
  \tfrac{2x-(b-a)c_i^n-(a+b)}{2t-(b-a)c_i^n-(a+b)}\right]\,dx
  &\colon (a<b) \text{ and }\big(\frac{2t-(a+b)}{b-a}\in\{c_1^n,\ldots,c_n^n\}\big)\\
  0&\colon\text{else,}
  \end{cases}
\end{split}     \end{equation}
let $(\bar{q}^{n,Q})_{n,Q\in\N_0}\subseteq \mathcal{Q}_T$ satisfy for all $n,Q\in\N$, $t\in[0,T]$
that $\bar{q}^{0,Q}(t)=\1_{\{0\}}(t)$
and
\begin{equation}\label{eq:def.bar.q}
\bar{q}^{n,Q}(t)=
  \sum_{s\in[0,t]}
  \bar{q}^{n-1,Q}(s)
  \,q^{Q,[s,T]}(t),
\end{equation}
let
$ 
  ({ U}_{ n,M,Q}^{\theta })_{n,M,Q\in\Z,\theta\in\Theta}
  \subseteq\mathcal{M}(\mathcal{B}([0,T]\times\R^d)\otimes\mathcal{F},\mathcal{B}(\R))
$
satisfy
for all 
$
  n,M,Q \in \N
$,
$ \theta \in \Theta $,
$ (t,x) \in [0,T]\times \R^d $
that $
U_{0,M,Q}^{\theta}(t,x)=0$ and
\begin{multline}\label{eq:def:U}
  U_{n,M,Q}^{\theta}(t,x)=
  \frac{1}{M^n}\sum_{i=1}^{M^n}g(x+W^{(\theta,0,-i)}_T-W^{(\theta,0,-i)}_t)
  \\
  +\sum_{l=0}^{n-1}\sum_{i=1}^{M^{n-l}}\sum_{s\in[t,T]}\frac{q^{Q,[t,T]}(s)}{M^{n-l}}
  \big(\funcF(U_{l,M,Q}^{(\theta,l,i,s)})-\1_{\N}(l)\funcF( U_{l-1,M,Q}^{(\theta,-l,i,s)})\big)
  (s,x+W_{s}^{(\theta,l,i)}-W_t^{(\theta,l,i)}),
 \end{multline}
for every $n,Q\in\N_0$ let
$\left\|\cdot\right\|_{n,Q}\colon \mathcal{M}(\mathcal{B}([0,T]\times\R^d)\otimes\mathcal{F},\mathcal{B}(\R))\to[0,\infty]$
be the function which satisfies
\begin{equation}  \begin{split}\label{eq:seminorms}
  &\left\|V\right\|_{n,Q}=
  \sum_{t\in[0,T]}\bar{q}^{n,Q}(t)
  \left[\sup_{s\in[t,T]}\sup_{u\in[0,s]}\sup_{z\in\R^d}\sqrt{\E\big[|V(s,z+W_u^0)|^2\big]}\right]
\end{split}     \end{equation}
for all
$V\in \mathcal{M}(\mathcal{B}([0,T]\times \R^d)\otimes\mathcal{F},\mathcal{B}(\R))$.

\subsection{Pseudocode}\label{sec:code}

In this subsection a mathematical style pseudocode illustrates that the multilevel Picard approximations~\eqref{eq:def:U}
can be easily implemented.
We assume that
the time horizon $T\in(0,\infty)$,
the dimension $d\in\N$,
the terminal condition $g\colon \R^d \to \R$,
the (gradient-independent) nonlinearity $f\colon [0,T]\times \R^d \times \R \to \R$,
the basis for the number of Monte-Carlo samples $M\in \N$,
the number of quadrature nodes $Q \in \N$,
increasingly ordered roots $c\in [-1,1]^Q$ of the $Q$-th Legendre polynomial,
and the corresponding Legendre quadrature weights $w\in [0,\infty)^Q$
are global variables.
For an implementation in {\sc Matlab} see~\cite{EHutzenthalerJentzenKruse2017}.
\begin{algorithm}[H]
\caption{Multilevel Picard approximation}
\begin{algorithmic}[1]
\Function{MLP}{$n,t,x}$
	\State $c_{loc}\leftarrow (T-t)c/2 +(T+t)/2$;
	\Comment Quadrature nodes on $[t,T]$
	\State $d\leftarrow c_{loc}-[t;c_{loc}(1:(Q-1))]$;
	\Comment Increments between consecutive quadrature nodes
	\State $w_{loc}\leftarrow (T-t)w/2$;
	\Comment Quadrature weights on $[t,T]$
	\State \multiline{Generate $M^n$ realizations $W(i)\in \R^d$, $i\in \{1,\ldots, M^n\}$, of independent standard normally distributed random vectors;}
	\State $u\leftarrow \frac{1}{M^n }\sum_{i=1}^{M^n}g(x+\sqrt{T-t}W(i))$;
	\For{$l\leftarrow 0$ to $(n-1)$}
		\State $X(i)\leftarrow x$ for all $i\in \{1,\ldots, M^{n-l}\}$;
		\For{$k\leftarrow 1$ to $Q$}
			\State \multiline{%
            Generate $M^{n-l}$ realizations $W(i)\in \R^d$, $i\in \{1,\ldots, M^{n-l}\}$, of independent standard normally distributed random vectors;}
			\State $X(i)\leftarrow X(i)+\sqrt{d(k)}W(i)$ for all $i\in \{1,\ldots, M^{n-l}\}$;
			\State $
			u\leftarrow u+\frac{w_{loc}(k)}{M^{n-l}}\sum_{i=1}^{M^{n-l}}
				f(c_{loc}(k), X(i), \text{MLP}(l,c_{loc}(k),X(i)))$;
			\If{$l>0$}
				\State $u\leftarrow u-\frac{w_{loc}(k)}{M^{n-l}}\sum_{i=1}^{M^{n-l}}
				f(c_{loc}(k), X(i), \text{MLP}(l-1,c_{loc}(k),X(i)))$;
			\EndIf
		\EndFor
	\EndFor
	\State \Return $u$;
\EndFunction
\end{algorithmic}
\end{algorithm}

\subsection{Sketch of the proof}

Throughout this subsection assume the setting in Subsection~\ref{sec:setting.full.discretization} and let $N,M,Q \in \N$. Theorem \ref{thm:rate} provides an upper bound for the distance between the approximation $U^0_{N,M,Q}$ and the PDE solution $u^\infty$ measured in the semi-norms $\|\cdot \|_{n,Q}$, $n \in \N_0$, given in \eqref{eq:seminorms}. We establish this bound by splitting the global error
$\|U_{N,M,Q}^{0}-u^{\infty}\|_{n,Q}$ 
 into the Monte Carlo error $\|U_{N,M,Q}^{0}-\E[U_{N,M,Q}^{0}]\|_{n,Q}$ and the time discretization error $\|\E[U_{N,M,Q}^{0}]-u^{\infty}\|_{n,Q}$. To analyze the time discretization error, we employ the Feynman-Kac formula to obtain 
 \begin{equation}  \begin{split}\label{eq:FeynmanKacSketch}
      u^{\infty}(s,x)=\E\!\left[g(x+W_{T-s}^0)
      +
      \int_s^{T}(\funcF(u^{\infty}))(t,x+W_{t-s}^0)\,dt
      \right]
    \end{split}     \end{equation}
    for all $s \in [0,T], x\in \R^d$
(see Lemma \ref{l:nonlinear.FK.formula} below). Moreover, the approximations admit the following Feynman-Kac-type representation
\begin{equation}  \begin{split}\label{eq:discreteFeynmanKacSketch}
  \E\!\left[U_{N,M,Q}^{0}(s,x)
    \right]
    =
    \E\!\left[g(x+W_{T-s}^0)+\sum_{t\in[s,T]}q^{Q,[s,T]}(t)
    \big(\funcF( U_{N-1,M,Q}^{0})\big)(t,x+W_{t-s}^0)
    \right]
\end{split}     \end{equation}
 for all $s \in [0,T], x\in \R^d$
(see Lemma \ref{l:approximations.integrable} below). This, \eqref{eq:FeynmanKacSketch} and the Lipschitz-type assumption \eqref{eq:fLipschitz} show that the time discretization error is bounded from above by the error of the $(N-1)$-th approximation
 $\|
     U_{N-1,M,Q}^{0}- u^{\infty}
    \|_{n+1,Q}
    $
     and the error of the Gau\ss-Legendre quadrature rule applied to the function $[s,T]\ni t \mapsto \E[\funcF(u^\infty)(t,x+W_{t-s})] \in \R$ (see \eqref{eq:final.Timeerror} below). 
     Combining this with 
the established bound for the Monte Carlo error (see \eqref{eq:MC.error} below) results in the recursive inequality for the global error \eqref{eq:global.estimate} that can be handled using a discrete Gronwall-type inequality. The error representation for Gau\ss-Legendre quadrature rules allows to further simplify the global error under suitable regularity assumptions (see Corollary \ref{c:rate3} below). In Section \ref{sec:comp_cost} we provide upper bounds for the number of realizations of scalar standard normal random variables and for the number of function evaluations of $f$ and $g$ required to compute one realization of $U^0_{N,M,Q}(t,x)$ for a single point $(t,x)\in [0,T]\times \R^d$ in space-time.
This and Corollary \ref{c:rate3} prove
in the case of the semilinear heat equation~\eqref{eq:PDE}
that the computational complexity 
(see Corollary \ref{c:rate3} for the precise definition hereof) of our proposed scheme grows linearly in the space
 dimension $d$ and polynomially in the inverse accuracy $\eps^{-1}$ under suitable assumptions (see Corollary \ref{c:rate4a} below).

\subsection{Preliminary results for the Gau\ss-Legendre quadrature rules}
\begin{lemma}[Gau\ss-Legendre over different intervals]\label{l:scaling.LG}
 Assume the setting in Subsection~\ref{sec:setting.full.discretization},
  let $n\in\N$, $s\in[0,T)$, $t\in[0,s]$,
  and let $\psi\colon[0,T]\to[0,\infty]$ be  a non-increasing function.
  Then we have
  \begin{equation}  \begin{split}
    \sum_{r\in[s,T]}q^{n,[s,T]}(r)\,\psi(r)
    \leq
    \sum_{r\in[t,T]}q^{n,[t,T]}(r)\,\psi(r).
  \end{split}     \end{equation}
\end{lemma}
\begin{proof}
Note that \eqref{eq:def_gauss_leg} and the integral transformation theorem with the substitution 
$[s,T]\ni x\mapsto (x-s)\tfrac{T-t}{T-s}+t\in[t,T]$ show that
 \begin{equation}  \begin{split}\label{eq:subs_q}
    &\sum_{r\in[s,T]}q^{n,[s,T]}(r)\,\psi(r)
    =\sum_{i=1}^{n}
    q^{n,[s,T]}(\tfrac{T-s}{2}c_i^n+\tfrac{T+s}{2})\,\psi(\tfrac{T-s}{2}c_i^n+\tfrac{T+s}{2})
    \\&
    =
    \sum_{i=1}^{n}
  \Bigg[\int_s^T\Bigg(\prod_{j\in\{1,\ldots,n\}\setminus\{i\}}
  \tfrac{2(x-s)-(T-s)c_j^n-(T-s)}{(T-s)c_i^n+(T+s)-(T-s)c_j^n-(T+s)}\Bigg)\,dx\Bigg]
    \psi(\tfrac{T-s}{2}c_i^n+\tfrac{T+s}{2})
    \\&
   =
    \sum_{i=1}^{n}
   \tfrac{T-s}{T-t}
 \Bigg[ \int_t^T\Bigg(\prod_{j\in\{1,\ldots,n\}\setminus\{i\}}
  \tfrac{2\frac{y-t}{T-t}-c_j^n-1}{c_i^n-c_j^n}\Bigg)\,dy\Bigg]
    \psi(\tfrac{T-s}{2}c_i^n+\tfrac{T+s}{2})
    \\&
   =
    \sum_{i=1}^{n}
   \tfrac{T-s}{T-t}
  \Bigg[\int_t^T\Bigg(\prod_{j\in\{1,\ldots,n\}\setminus\{i\}}
  \tfrac{2(y-t)-(T-t)c_j^n-(T-t)}{(T-t)c_i^n+(T+t)-(T-t)c_j^n-(T+t)}\Bigg)\,dy\Bigg]
    \psi(\tfrac{T-s}{2}c_i^n+\tfrac{T+s}{2})
    \\&
    =
    \sum_{i=1}^{n}
  \tfrac{T-s}{T-t}\,
    q^{n,[t,T]}(\tfrac{T-t}{2}c_i^n+\tfrac{T+t}{2})\,
    \psi(\tfrac{T-s}{2}c_i^n+\tfrac{T+s}{2}).
  \end{split}     \end{equation}
Observe that the fact that $t\le s$ and the fact that
$\forall\,  i \in\{1,\ldots,n\}\colon c_i^{n}\in[-1,1]$
ensure that for all $i\in\{1,\ldots,n\}$ 
it holds that $\tfrac{T-s}{2}c_i^n+\tfrac{T+s}{2}\geq 
   \tfrac{T-t}{2}c_i^n+\tfrac{T+t}{2}$. This and the fact
   that $\psi$ is non-increasing imply for all $i\in\{1,\ldots,n\}$ 
   that $\psi(\tfrac{T-s}{2}c_i^n+\tfrac{T+s}{2})\leq 
   \psi(\tfrac{T-t}{2}c_i^n+\tfrac{T+t}{2})$.
Combining this with \eqref{eq:subs_q}, \eqref{eq:def_gauss_leg}, and the fact that $\tfrac{T-s}{T-t}\leq 1$ proves that
  \begin{equation}  \begin{split}
    &\sum_{r\in[s,T]}q^{n,[s,T]}(r)\,\psi(r)
    \leq
    \sum_{i=1}^{n}
    q^{n,[t,T]}(\tfrac{T-t}{2}c_i^n+\tfrac{T+t}{2})\,\psi(\tfrac{T-t}{2}c_i^n+\tfrac{T+t}{2})
    =
    \sum_{r\in[t,T]}q^{n,[t,T]}(r)\,\psi(r).
  \end{split}     \end{equation}
\end{proof}
\begin{lemma}\label{l:quad.rule.rn}
 Assume the setting in Subsection~\ref{sec:setting.full.discretization}
  and
  let $Q\in\N$.
  Then, for all $n\in\N_0$, $k\in \N_0\cap [0,2Q-n)$ we have
  \begin{equation}  \begin{split}\label{eq:quad.rule.n}
    \sum_{t\in[0,T]}\bar{q}^{n,Q}(t)\,\tfrac{(T-t)^k}{k!}=\tfrac{T^{n+k}}{(n+k)!}.
  \end{split}     \end{equation}
\end{lemma}
\begin{proof}
\sloppy{First, note that the fact that the 
  Gau\ss-Legendre quadrature rule $C([0,T],\R)\ni \varphi \mapsto \sum_{t\in [0,T]}q^{Q,[0,T]}(t)\varphi(t)\in \R$ integrates polynomials of order less than $2Q$ exactly
  implies that for all}
  $s\in[0,T]$, $k\in \N_0 \cap [0, 2Q)$ it holds that
  \begin{equation}  \begin{split}\label{eq:GL.exact}
    \sum_{t\in[s,T]}q^{Q,[s,T]}(t)\,\tfrac{(T-t)^k}{k!}=\int_s^T\tfrac{(T-t)^k}{k!}\,dt=\tfrac{(T-s)^{k+1}}{(k+1)!}.
  \end{split}     \end{equation}
  We now prove \eqref{eq:quad.rule.n} by induction on $n\in \N_0$. For the base case $n=0$ we note that for all $k\in \N_0$ it holds that
  \begin{equation}
  \sum_{t\in[0,T]}\bar{q}^{0,Q}(t)\,\tfrac{(T-t)^k}{k!}= \sum_{t\in[0,T]}\1_{\{0\}}(t)\,\tfrac{(T-t)^k}{k!}=\tfrac{T^{k}}{k!}.
  \end{equation}
  This establishes \eqref{eq:quad.rule.n} in the base case $n=0$. For the induction step $\N_0 \ni n \rightarrow n+1 \in \N$ we observe that \eqref{eq:GL.exact} and the induction hypothesis imply that for all $k\in \N_0\cap [0,2Q-n-1)$ it holds that
  \begin{equation}  \begin{split}
    \sum_{t\in[0,T]}\bar{q}^{n+1,Q}(t)\,\tfrac{(T-t)^k}{k!}
    &=
    \sum_{t\in[0,T]}\bigg[\sum_{s\in[0,t]}\bar{q}^{n,Q}(s)\,q^{Q,[s,T]}(t)\bigg]\tfrac{(T-t)^k}{k!}
    =
    \sum_{s\in[0,T]}\bar{q}^{n,Q}(s)\bigg[\sum_{t\in[s,T]}q^{Q,[s,T]}(t)\,\tfrac{(T-t)^k}{k!}\bigg]
    \\&
    =
    \sum_{s\in[0,T]}\bar{q}^{n,Q}(s)\,\tfrac{(T-s)^{k+1}}{(k+1)!}
    =\tfrac{T^{n+1+k}}{(n+1+k)!}.
  \end{split}     \end{equation}
  This finishes the induction step $\N_0 \ni n \rightarrow n+1 \in \N$. Induction hence establishes \eqref{eq:quad.rule.n}. The proof of Lemma~\ref{l:quad.rule.rn} is thus completed.
\end{proof}
\subsection{Preliminary results for the semi-norms}
We refer to a $[0,\infty]$-valued function as semi-norm
if it is subadditive and absolutely homogeneous. In particular, we do not require
semi-norms to have finite values.
The proof of the following lemma is clear and therefore omitted.
\begin{lemma}[Seminorm property]\label{l:seminorm}
  Assume the setting in Subsection~\ref{sec:setting.full.discretization} and
  let $k\in\N_0$.
  Then the function 
  $\mathcal{M}(\mathcal{B}([0,T]\times \R^d)\otimes\mathcal{F},\mathcal{B}(\R))
  \ni U\mapsto \|U\|_{k,Q}\in[0,\infty]$ is a semi-norm
  in the sense that it is subadditive, nonnegative, and absolutely homogeneous.
\end{lemma}
The following lemma implies 
that Monte Carlo averages converge in our semi-norms with rate $1/2$.
\begin{lemma}[Linear combinations of iid random variables]\label{l:BDG}
  Assume the setting in Subsection~\ref{sec:setting.full.discretization},
  let $k\in\N_0$, $n,Q\in\N$, $r_1,\ldots,r_n\in\R$, and let $V_1,\ldots,V_n\in 
  \mathcal{M}(\mathcal{B}([0,T]\times \R^d)\otimes\mathcal{F},\mathcal{B}(\R))$
  satisfy for all $(s,x)\in[0,T]\times\R^d$
  that $V_1(s,x),\ldots,V_n(s,x)$ are integrable random variables
  which are
  independent and identically distributed and which are independent of $W^0$.
  Then
  \begin{equation}  \begin{split}
    \left\|\sum_{i=1}^nr_i(V_i-\E[V_i])\right\|_{k,Q}
    = \left\|(V_1-\E[V_1])\right\|_{k,Q}
    \sqrt{\sum_{i=1}^n|r_i|^2}
    \leq \left\|V_1\right\|_{k,Q}
    \sqrt{\sum_{i=1}^n|r_i|^2}.
  \end{split}     \end{equation}
\end{lemma}
\begin{proof}
  The definition~\eqref{eq:seminorms} of the semi-norm
  and the fact that for all $(s,x)\in[0,T]\times \R^d$ it holds
  that $(V_i(s,x))_{i\in\{1,\ldots,n\}}$ are independent of $W^0$ and are
  independent and identically distributed 
  imply that
  \begin{equation}  \begin{split}
    &\left\|\sum_{i=1}^nr_i(V_i-\E[V_i])\right\|_{k,Q}
    \\&=
  \sum_{t\in[0,T]}\bar{q}^{k,Q}(t)
  \left[\sup_{s\in[t,T]}\sup_{u\in[0,s]}\sup_{z\in\R^d}
  \E\!\left[\E\!\left[\bigg|\sum_{i=1}^nr_i\Big(V_i(s,z+W_u^0)
                                                 -
                                                 \E\left[V_i(s,z+W_u^0)|W^0\right]\Big)\bigg|^2\,\Bigg|
  W^0\right]\right]\right]^{\frac{1}{2}}
    \\&=
  \sum_{t\in[0,T]}\bar{q}^{k,Q}(t)
  \left[\sup_{s\in[t,T]}\sup_{u\in[0,s]}\sup_{z\in\R^d}
  \E\!\left[\Var\!\left(\sum_{i=1}^nr_iV_i(s,z+W_u^0)|W^0\right)\right]\right]^{\frac{1}{2}}
    \\&=
  \sum_{t\in[0,T]}\bar{q}^{k,Q}(t)
  \left[\sup_{s\in[t,T]}\sup_{u\in[0,s]}\sup_{z\in\R^d}
  \E\left[\sum_{i=1}^n|r_i|^2\Var\!\left(V_1(s,z+W_u^0)\Big|W^0\right)\right]\right]^{\frac{1}{2}}
    \\&=\left\|V_1-\E[V_1]\right\|_{k,Q}
    \sqrt{\sum_{i=1}^n|r_i|^2}
    \\&\leq
  \sum_{t\in[0,T]}\bar{q}^{k,Q}(t)
  \left[\sup_{s\in[t,T]}\sup_{u\in[0,s]}\sup_{z\in\R^d}
  \E\!\left[\E\!\left[\Big|V_1(s,z+W_u^0)\Big|^2\Big|W^0\right]\right]\right]^{\frac{1}{2}}
    \sqrt{\sum_{i=1}^n|r_i|^2}
  \\&=\left\|V_1\right\|_{k,Q}
    \sqrt{\sum_{i=1}^n|r_i|^2}.
  \end{split}     \end{equation}
\end{proof}
\begin{lemma}[Lipschitz property]\label{l:FcircNorm2}
  Assume the setting in Subsection~\ref{sec:setting.full.discretization},
  let $k\in\N_0$, $Q\in\N$, and let $U,V\in
  \mathcal{M}(\mathcal{B}([0,T]\times \R^d)\otimes\mathcal{F},\mathcal{B}(\R))$.
  Then
  \begin{equation}  \begin{split}
    &\|F(U)-F(V)\|_{k,Q} \leq L \|U-V\|_{k,Q}.
  \end{split}     \end{equation}
\end{lemma}
\begin{proof}
  The definition~\eqref{eq:seminorms} of the semi-norm
  and the global Lipschitz property~\eqref{eq:fLipschitz} of $\funcF$
  imply that
  \begin{equation}  \begin{split}
    \|F(U)-F(V)\|_{k,Q}
  &=
  \sum_{t\in[0,T]}\bar{q}^{k,Q}(t)
  \left[\sup_{s\in[t,T]}\sup_{u\in[0,s]}\sup_{z\in\R^d}\E\big[|(F(U))(s,z+W_u^0)-(F(V))(s,z+W_u^0)|^2\big]\right]^{\frac{1}{2}}
  \\&\leq
  \sum_{t\in[0,T]}\bar{q}^{k,Q}(t)
  \left[\sup_{s\in[t,T]}\sup_{u\in[0,s]}\sup_{z\in\R^d}L^2\E\big[|U(s,z+W_u^0)-V(s,z+W_u^0)|^2\big]\right]^{\frac{1}{2}}
    \\&=
     L \|U-V\|_{k,Q}.
  \end{split}     \end{equation}
\end{proof}
\begin{lemma}\label{l:NormOfInt}
  Assume the setting in Subsection~\ref{sec:setting.full.discretization},
  let $k\in\N_0$, $Q\in\N$, and let $U\in
  \mathcal{M}(\mathcal{B}([0,T]\times \R^d)\otimes\mathcal{F},\mathcal{B}(\R))$
  satisfy for all $(s,x)\in[0,T]\times\R^d$ that $U(s,x)$ and $W^0$ are independent.
  Then
  \begin{equation}  \begin{split}
    \left\|[0,T]\times \R^d \ni (s,z)\mapsto 
    \sum_{t\in[s,T]}q^{Q,[s,T]}(t) U(t,z+W_t^0-W_s^0)\in\R \right\|_{k,Q} \leq \|U\|_{k+1,Q}
  \end{split}     \end{equation}
\end{lemma}
\begin{proof}
  The definition~\eqref{eq:seminorms} of the semi-norm,
  the triangle inequality,
  independence,
  Lemma~\ref{l:scaling.LG},
  and
  the definition~\eqref{eq:def.bar.q} of $\bar{q}^{k+1,Q}$
  yield that
  \begin{equation}  \begin{split}
    &\bigg\|[0,T]\times \R^d \ni (s,z)\mapsto 
    \sum_{r\in[s,T]}q^{Q,[s,T]}(r) U(r,z+W_r^0-W_s^0)\in\R \bigg\|_{k,Q}
    \\&=
  \sum_{t\in[0,T]}\bar{q}^{k,Q}(t)
  \left[\sup_{s\in[t,T]}\sup_{u\in[0,s]}\sup_{z\in\R^d}
  \left(\E\bigg[\Big|\sum_{r\in[s,T]}q^{Q,[s,T]}(r) U(r,z+W_u^{0}+W_r^0-W_s^0)
  \Big|^2\bigg]\right)^{\frac{1}{2}}\right]
    \\&\leq
  \sum_{t\in[0,T]}\bar{q}^{k,Q}(t)
  \left[\sup_{s\in[t,T]}\sup_{u\in[0,s]}\sup_{z\in\R^d}
  \sum_{r\in[s,T]}q^{Q,[s,T]}(r)\left(\E\bigg[\Big| U(r,z+W_u^{0}+W_r^0-W_s^0)
  \Big|^2\bigg]\right)^{\frac{1}{2}}\right]
    \\&\leq
  \sum_{t\in[0,T]}\bar{q}^{k,Q}(t)\sup_{s\in[t,T]}\sum_{r\in[s,T]}q^{Q,[s,T]}(r)
  \left[\sup_{v\in[r,T]}\sup_{u\in[0,v]}\sup_{z\in\R^d}
  \E\bigg[\Big| U(v,z+W_u^{0})
  \Big|^2\bigg]\right]^{\frac{1}{2}}
    \\&=
  \sum_{t\in[0,T]}\bar{q}^{k,Q}(t)\sum_{r\in[t,T]}q^{Q,[t,T]}(r)
  \left[\sup_{v\in[r,T]}\sup_{u\in[0,v]}\sup_{z\in\R^d}
  \E\bigg[\Big| U(v,z+W_u^{0})
  \Big|^2\bigg]\right]^{\frac{1}{2}}
    \\&=
  \sum_{r\in[0,T]}\bar{q}^{k+1,Q}(r)
  \left[\sup_{v\in[r,T]}\sup_{u\in[0,v]}\sup_{z\in\R^d}
  \E\bigg[\Big| U(v,z+W_u^{0})
  \Big|^2\bigg]\right]^{\frac{1}{2}}
    \\&= \|U\|_{k+1,Q}.
  \end{split}     \end{equation}
\end{proof}

\begin{lemma}[Monotonicity]\label{l:NormMonotonicity}
  Assume the setting in Subsection~\ref{sec:setting.full.discretization},
  let $k\in\N_0$, $Q\in\N$, let $\mathcal{G}\subseteq\mathcal{F}$ be a $\sigma$-algebra, and let $U,V\in
  \mathcal{M}(\mathcal{B}([0,T]\times \R^d)\otimes\mathcal{F},\mathcal{B}(\R))$
  satisfy $|U|\leq |V|$.
  Then
  \begin{equation}  \begin{split}
    \left\|\E\big[|U|\big|\mathcal{G}\big]\right\|_{k,Q}
    \leq
    \left\|U\right\|_{k,Q}
    \leq
    \left\|V\right\|_{k,Q}.
  \end{split}     \end{equation}
\end{lemma}
\begin{proof}
  The definition~\eqref{eq:seminorms} of the semi-norm,
  Jensen's inequality,
  and the hypothesis that $|U|\leq |V|$ imply that
  \begin{equation}  \begin{split}
    \left\|\E\big[|U|\big|\mathcal{G}\big]\right\|_{k,Q}
    &=
  \sum_{t\in[0,T]}\bar{q}^{k,Q}(t)
  \left[\sup_{s\in[t,T]}\sup_{u\in[0,s]}\sup_{z\in\R^d}
  \left(\E\bigg[\Big|\E\big[|U(s,z+W_u^{0})|\big|\mathcal{G}\big]
  \Big|^2\bigg]\right)^{\frac{1}{2}}\right]
    \\&\leq
  \sum_{t\in[0,T]}\bar{q}^{k,Q}(t)
  \left[\sup_{s\in[t,T]}\sup_{u\in[0,s]}\sup_{z\in\R^d}
  \left(\E\bigg[\Big|U(s,z+W_u^{0})
  \Big|^2\bigg]\right)^{\frac{1}{2}}\right]
  =
    \left\|U\right\|_{k,Q}
    \\&\leq
  \sum_{t\in[0,T]}\bar{q}^{k,Q}(t)
  \left[\sup_{s\in[t,T]}\sup_{u\in[0,s]}\sup_{z\in\R^d}
  \left(\E\bigg[\Big|V(s,z+W_u^{0})
  \Big|^2\bigg]\right)^{\frac{1}{2}}\right]
  =
    \left\|V\right\|_{k,Q}.
  \end{split}     \end{equation}
\end{proof}
The following lemma specifies the values of our semi-norms of constant functions.
It follows directly from
the definition~\eqref{eq:seminorms} of the semi-norms
and from
Lemma~\ref{l:quad.rule.rn}. Its proof is therefore omitted.
\begin{lemma}[Seminorm of constants]\label{l:NormOfConstants}
  Assume the setting in Subsection~\ref{sec:setting.full.discretization}
  and
  let $Q\in\N$, $k\in\N_0\cap[0,2Q-1]$.
  Then $\left\|1\right\|_{k,Q}=\tfrac{T^k}{k!}$.
\end{lemma}

\subsection{Error analysis for multilevel Picard iteration}
\begin{lemma}[Approximations are integrable]\label{l:approximations.integrable}
 Assume the setting in Subsection~\ref{sec:setting.full.discretization},
let $z\in\R^d$, $M,Q\in\N$, and assume for all $s\in[0,T]$, $t\in [s,T]$ that $\E\big[|g(z+W_t^0)|+|(\funcF(0))(t,z+W_s^0)|\big]<\infty$.
Then 
\begin{enumerate} [(i)] 
\item \label{item:approximations.integrable.i} for all
$n\in\N_0$, $\theta\in\Theta$, $s\in[0,T]$, $t\in [s,T]$ it holds that
\begin{equation} \label{eq:approximations.integrable} \begin{split}
  \E\Big[\big|U_{n,M,Q}^{\theta}(t,z+W_s^0)\big|
+\big|
    \big(\funcF( U_{n,M,Q}^{\theta})\big)(t,z+W_{s}^0)
\big|
\Big]<\infty
\end{split}     \end{equation}
and
\item \label{item:approximations.integrable.ii} for all 
$n\in\N_0$, $\theta\in\Theta$, $s\in[0,T]$ it holds that
\begin{equation}  \begin{split}\label{eq:discreteFeynmanKac}
  \E\!\left[U_{n+1,M,Q}^{\theta}(s,z)
    \right]
    =
    \E\!\left[g(z+W_{T-s}^0)+\sum_{t\in[s,T]}q^{Q,[s,T]}(t)
    \big(\funcF( U_{n,M,Q}^{\theta})\big)(t,z+W_{t-s}^0)
    \right].
\end{split}     \end{equation}
\end{enumerate}
\end{lemma}
\begin{proof}
  We prove \eqref{item:approximations.integrable.i} by induction on $n\in \N_0$. For the base case $n=0$ we note that for all
$\theta\in\Theta$, $s\in[0,T]$, $t\in [s,T]$ it holds that
\begin{equation}  \begin{split}
  \E\Big[\big|U_{0,M,Q}^{\theta}(t,z+W_s^0)\big|
+\big|
    \big(\funcF( U_{0,M,Q}^{\theta})\big)(t,z+W_{s}^0)
\big|
\Big]=\E\Big[\big|
    ((\funcF( 0))(t,z+W_{s}^0)
\big|
\Big]<\infty.
 \end{split}     \end{equation}
 This establishes \eqref{item:approximations.integrable.i} in the base case $n=0$. For the induction step $\N_0 \ni n \rightarrow n+1 \in \N$
 let $n\in \N_0$ and assume that \eqref{item:approximations.integrable.i} holds for $n=0$, $n=1$, $\ldots$, $n=n$. 
 The induction hypothesis and~\eqref{eq:def:U}
  imply that for all $\theta\in\Theta$, $s\in[0,T]$, $t\in[s,T]$ it holds that
\begin{multline}
  \E\Big[\big|U_{n+1,M,Q}^{\theta}(t,z+W_s^0)\big|
\Big]
\leq \E\Big[\big|g(z+W^0_{T-t+s})\big|\Big]
\\
+\sum_{l=0}^{n}\sum_{r\in[t,T]}\frac{q^{Q,[t,T]}(r)}{M^{n+1-l}}\sum_{i=1}^{M^{n+1-l}}\sum_{k\in\{l-1,l\}\cap\N_0}\max_{j\in\{-l,l\}}
  \E\!\left[\big|\big(\funcF(U_{k,M,Q}^{(\theta,j,i,r)})\big)(r,z+W^0_{s+r-t})\big|\right]<\infty.
\end{multline}
Combining this with \eqref{eq:fLipschitz} proves for all $\theta\in\Theta$, $s\in[0,T]$, $t\in[s,T]$ that
\begin{equation} \begin{split}
  &\E\Big[
    \big|\big(\funcF( U_{n+1,M,Q}^{\theta})\big)(t,z+W_{s}^0)
\big|
\Big]
\\&\leq \E\Big[
    \big|\big(\funcF( U_{n+1,M,Q}^{\theta})\big)(t,z+W_{s}^0)
-\big(\funcF( 0)\big)(t,z+W_{s}^0)\big|
\Big]+\E\Big[
    \big|\big(\funcF( 0)\big)(t,z+W_{s}^0)
\big|
\Big]\\
&\leq L\E\Big[
    \big| U_{n+1,M,Q}^{\theta}(t, z+W_{s}^0)
\big|
\Big]+\E\Big[
    \big|\big(\funcF( 0)\big)(t,z+W_{s}^0)
\big|
\Big]
<\infty.
\end{split}     \end{equation}
This finishes the induction step $\N_0 \ni n \rightarrow n+1 \in \N$. Induction hence 
establishes~\eqref{item:approximations.integrable.i}.
  Next we note that \eqref{eq:def:U}, the fact that 
  $(U_{n,M,Q}^{\theta})_{n\in\N_0}$, $\theta\in\Theta$, are identically
distributed, and a telescope argument yield that
  for all $n\in\N_0$, $\theta\in\Theta$, $s\in[0,T]$
  it holds that
  \begin{equation}  \begin{split}
    &\E\!\left[U_{n+1,M,Q}^{\theta}(s,z))\right]
    -\E\!\left[g(z+W_{T-s}^{0})\right]
    \\&
    =\sum_{l=0}^{n}\sum_{v\in[s,T]}q^{Q,[s,T]}(v)\,
    \E\Big[\!
    \left(\funcF( U_{l,M,Q}^{0})-\1_{\N}(l)\funcF( U_{l-1,M,Q}^{0})\right)\!(v,z+W_v^0-W_s^0)
    \Big]
    \\&
    =
    \sum_{v\in[s,T]}q^{Q,[s,T]}(v)\,
    \E\Big[\!\left(\funcF( U_{n,M,Q}^{0})\right)\!(v,z+W_v^0-W_s^0)
    \Big]
    \\&
    =\E\!\left[
    \sum_{v\in[s,T]}q^{Q,[s,T]}(v)
    \left(\funcF( U_{n,M,Q}^{\theta})\right)\!(v,z+W_{v-s}^0)
    \right].
  \end{split}     \end{equation}
  This establishes \eqref{item:approximations.integrable.ii}. The proof of Lemma \ref{l:approximations.integrable} is thus completed.
\end{proof}
\begin{lemma}[Nonlinear Feynman-Kac formula]\label{l:nonlinear.FK.formula}
 Assume the setting in Subsection~\ref{sec:setting.full.discretization},
let $z\in\R^d$, and assume for all $s\in[0,T]$ that
\begin{equation}  \begin{split}\label{eq:ldeltaintegrand3}
  \E\!\left[ \sup_{t\in[s,T]}\left| u^{\infty}(t,z+W_{t-s}^0)\right| +\int_s^T\left|\big(\funcF(0)\big)(t,z+W_{t-s}^0)\right|\,dt
         \right]<\infty.
\end{split}     \end{equation}
Then 
\begin{enumerate} [(i)] 
\item \label{item:approximations.integrable.iFC}
  for all $s\in[0,T]$ it holds that
  \begin{equation}  \begin{split}\label{eq:approximations.integrable.iFC}
  \E\left[
  \sup_{t\in[s,T]}
     | u^{\infty}(t,z+W_{t}^0-W_{s}^0)|
         +\int_s^T|(\funcF(u^{\infty}))(t,z+W_{t}^0-W_{s}^0)|\,dt
         \right]<\infty
  \end{split}     \end{equation}
and
\item \label{item:approximations.integrable.iiFC}
  for all $s\in[0,T]$ it holds that
  \begin{equation}  \begin{split}\label{eq:FeynmanKac}
      u^{\infty}(s,z)-\E\!\left[g(z+W_{T-s}^0)\right]
      &=\E\!\left[
      \int_s^{T}(\funcF(u^{\infty}))(t,z+W_{t-s}^0)\,dt
      \right].
    \end{split}     \end{equation}
\end{enumerate}
\end{lemma}
\begin{proof}
Note that \eqref{eq:fLipschitz} and \eqref{eq:ldeltaintegrand3} imply $(i)$.
 Next
 It\^o's formula and the PDE~\eqref{eq:PDE} ensure that
  for all $s\in[0,T]$, $t\in[s,T]$
  it holds $\P$-a.s.\ that
  \begin{align}
   \nonumber
   &
   u^{\infty}(t,z+W_t^0-W_s^0)-u^{\infty}(s,z)
   \\&=\int_s^t \left(
  \tfrac{\partial}{\partial r}u^{\infty}
+\tfrac{1}{2}\Delta_y u^{\infty}
  \right)\!(r,z+W_r^0-W_s^0)\,dr+\int_s^t\langle (\nabla_y u^{\infty})(r,z+W_r^0-W_s^0), \,dW_r^0\rangle 
  \label{eq:ItoFormula}
  \\&
   =-\int_s^t \left( \funcF(u^{\infty}) \right)\!(r,z+W_r^0-W_s^0)\,dr
   +\int_s^t\langle (\nabla_y u^{\infty})(r,z+W_r^0-W_s^0), \,dW_r^0\rangle .
   \nonumber
  \end{align}
 This and \eqref{eq:approximations.integrable.iFC} show that for all $s\in[0,T]$ it holds that
  $\E\big[\sup_{t\in[s,T]}\big|
   \int_s^t\langle (\nabla_y u^{\infty})(r,z+W_r^0-W_s^0), \,dW_r^0\rangle \big|\big]<\infty$.
  This ensures that
  $\E\big[
   \int_s^T\langle (\nabla_y u^{\infty})(t,z+W_t^0-W_s^0), \,dW_t^0\rangle \big]=0$.
  This
  and \eqref{eq:ItoFormula} prove 
  for all $s\in[0,T]$
  that
  \begin{equation}  \begin{split}
      u^{\infty}(s,z)-\E[g(z+W_{T-s}^0)]
      =
      u^{\infty}(s,z)-\E[u^{\infty}(T,z+W_{T}^0-W_s^0)]
      =\E\!\left[
      \int_s^{T}(\funcF(u^{\infty}))(t,z+W_{t-s}^0)\,dt
      \right].
  \end{split}     \end{equation}
  This finishes the proof of Lemma~\ref{l:nonlinear.FK.formula}.
\end{proof}
\begin{theorem}\label{thm:rate}
Assume the setting in Subsection~\ref{sec:setting.full.discretization},
let $M,Q\in \N$, $N\in \N \cap [1,2Q-1]$, $\theta \in\Theta$,
and
assume for all $z\in\R^d$, $s\in[0,T]$ that
\begin{equation}\label{eq:ldeltaintegrand}
  \E\!\left[ \sup_{t\in[s,T]}\left| u^{\infty}(t,z+W_{t-s}^0)\right| +\int_s^T\left|\big(\funcF(0)\big)(t,z+W_{t-s}^0)\right|\,dt
         \right]<\infty.
\end{equation}
 Then we have
 \begin{align}\label{eq:rate}
    &\left\|U_{N,M,Q}^{\theta}-u^{\infty}\right\|_{0,Q}
    \leq
    \left(1+2\LipConst\right)^{N-1}
    \bigg\{
     \LipConst \sup_{i\in\{1,2,\ldots,N\}}
      \tfrac{ \left\|u^{\infty}\right\|_{i,Q}}{\sqrt{M^{N-i}}}
   \\&
     +
     \Big[\sup_{i\in\{0,1,\ldots,N-1\}}
      \tfrac{ T^{i}}{i!\sqrt{M^{N-i}}}
      \Big]
     \left[
      \sup_{z\in\R^d}
     \sup_{s\in[0,T]}\left\|g(z+W_s^{0})\right\|_{L^2(\P;\R)}
     +T
      \sup_{z\in\R^d}
     \sup_{r,s\in[0,T]}\left\|(\funcF(0)) (r,z+W_s^0) \right\|_{L^2(\P;\R)}
     \right]
     \nonumber
   \\&
   +e^{T}
    \sup_{\substack{t\in[0,T],\\ r \in [0,t],\\ z\in\R^d}}
    \bigg\|
    \E\bigg[
    \sum\limits_{s\in[t,T]}q^{Q,[t,T]}(s)
    \left(\funcF(  u^{\infty})\right)\!(s,z+W_{r+s-t}^0)
      -
      \int_t^{T}
      \left(\funcF( u^{\infty})\right)\!(s,z+W_{r+s-t}^0)
      \,ds
    \, \Big | \, W_r^{0}
    \bigg]
      \bigg\|_{L^2(\P;\R)}
     \bigg\}.
     \nonumber
 \end{align}
\end{theorem}
\begin{proof}
Throughout this proof assume w.l.o.g.\ that the right-hand side of~\eqref{eq:rate} is finite, assume w.l.o.g.\ that $\theta=0$ (the case $\theta\neq 0$ follows from the case $\theta=0$),
  let $\eps\in [0,\infty)$ be the real number given
  by
  \begin{align}
  \nonumber
    \eps &=
    \sup_{\substack{t\in[0,T],\\ z\in\R^d}}\sup_{u\in[0,t]}
    \bigg\|
    \E\bigg[
    \sum_{s\in[t,T]}q^{Q,[t,T]}(s)
    \left(\funcF(  u^{\infty})\right)\!(s,z+W_{u+s-t}^0)
      -
      \int_t^{T}
      (\funcF( u^{\infty}))(s,z+W_{u+s-t}^0)
      \,ds
    \,\Big |\,W_u^{0}
    \bigg]
      \bigg\|_{L^2(\P;\R)},
  \end{align}
  and
  let 
  $
    (e_n)_{n\in\{0,1,\ldots,N\}}\subseteq [0,\infty]
  $
  be the extended real numbers which satisfy for all  $n\in\{0,1,\ldots,N\}$
  that
  \begin{equation}
    e_n
    =
    \sup\!\left\{
    \sqrt{M^{-j}}
    \left\|
      U_{n,M,Q}^{0} - u^{ \infty } 
    \right\|_{  k,Q }
    \colon
    k, j \in \N_0 ,
    k + j + n = N 
    \right\}.
  \end{equation}
 First, we analyze the \emph{Monte Carlo error}.
Item~\eqref{item:approximations.integrable.i} of Lemma~\ref{l:approximations.integrable} shows
for all $n\in\N_0$, $(t,z)\in[0,T]\times\R^d$, $s\in[0,t]$
that $\E[|U_{n,M,Q}^0(t,z+W^0_s)|]<\infty$.
 The triangle inequality,
 independence,
 Lemma~\ref{l:BDG},
 Lemma~\ref{l:NormMonotonicity},
 Lemma~\ref{l:NormOfInt},
 Lemma~\ref{l:NormOfConstants},
 and
 Lemma~\ref{l:FcircNorm2}
 imply that
 for all $n\in \N$, $k\in \N_0$ it holds that
  \begin{equation}  \begin{split}\label{eq:MC.error}
   & \left\|U_{n,M,Q}^{0}-\E\!\left[U_{n,M,Q}^{0}\right]\right\|_{k,Q}
   \\&
   \leq 
   \left\|[0,T]\times\R^d\ni (t,z)\mapsto M^{-n}\sum_{i=1}^{M^n}
   \left(g(z+W_T^{(\theta,0,-i)}-W_t^{(\theta,0,-i)})-\E\big[g(z+W_T^{(\theta,0,-i)}-W_t^{(\theta,0,-i)})\big]\right)\in\R\right\|_{k,Q}
   \\&\quad
   +\sum_{l=0}^{n-1}
   \bigg\|[0,T]\times\R^d\ni (t,z)\mapsto 
   \\&\quad\qquad
     M^{l-n}\sum_{i=1}^{M^{n-l}}\sum_{r\in[t,T]} q^{Q,[t,T]}(r)
     \big(\funcF(U_{l,M,Q}^{(\theta,l,i,r)})-\1_{\N}(l)\funcF( U_{l-1,M,Q}^{(\theta,-l,i,r)})\big)
  (r,z+W_{r}^{(\theta,l,i)}-W_t^{(\theta,l,i)})
   \\&\quad\qquad
  -
     M^{l-n}\sum_{i=1}^{M^{n-l}}\sum_{r\in[t,T]} q^{Q,[t,T]}(r)
     \E\left[
     \big(\funcF(U_{l,M,Q}^{(\theta,l,i,r)})-\1_{\N}(l)\funcF( U_{l-1,M,Q}^{(\theta,-l,i,r)})\big)
  (r,z+W_{r}^{(\theta,l,i)}-W_t^{(\theta,l,i)})
  \right]
   \in\R\bigg\|_{k,Q}
   \\&
   \leq \tfrac{1}{\sqrt{M^n}}
   \left\|[0,T]\times\R^d\ni (t,z)\mapsto g(z+W_T^0-W_t^0)\in\R\right\|_{k,Q}
   \\&\quad
   +\sum_{l=0}^{n-1}\tfrac{1}{\sqrt{M^{n-l}}}
   \bigg\|[0,T]\times\R^d\ni (t,z)\mapsto 
     \sum_{r\in[t,T]} q^{Q,[t,T]}(r)
   \\&\qquad\qquad\cdot
     \big(\funcF(U_{l,M,Q}^{(\theta,1,1,r)})-\1_{\N}(l)\funcF( U_{l-1,M,Q}^{(\theta,-1,1,r)})\big)
  (r,z+W_{r}^{0}-W_t^{0})
   \in\R\bigg\|_{k,Q}
   \\&
   \leq \tfrac{1}{\sqrt{M^n}}
   \sup_{z\in\R^d}\sup_{s\in[0,T]}\left\|g(z+W_s^0)\right\|_{L^2(\P;\R)}\|1\|_{k,Q}
   +\sum_{l=0}^{n-1}\tfrac{1}{\sqrt{M^{n-l}}}
   \left\|
     \funcF(U_{l,M,Q}^{(\theta,1,1,0)})-\1_{\N}(l)\funcF( U_{l-1,M,Q}^{(\theta,-1,1,0)})
   \right\|_{k+1,Q}
   \\&
   \leq \tfrac{1}{\sqrt{M^n}}
   \sup_{z\in\R^d}\sup_{s\in[0,T]}\left\|g(z+W_s^0)\right\|_{L^2(\P;\R)}\tfrac{T^k}{k!}
   +\tfrac{1}{\sqrt{M^n}}
   \left\|
     \funcF(0)
   \right\|_{k+1,Q}
   +L\sum_{l=1}^{n-1}\tfrac{1}{\sqrt{M^{n-l}}}
   \left\|
     U_{l,M,Q}^{(\theta,1,1,0)}- U_{l-1,M,Q}^{(\theta,-1,1,0)}
   \right\|_{k+1,Q}
   \\&
   \leq \tfrac{1}{\sqrt{M^n}}
   \sup_{z\in\R^d}\sup_{s\in[0,T]}\left\|g(z+W_s^0)\right\|_{L^2(\P;\R)}\tfrac{T^k}{k!}
   +\tfrac{1}{\sqrt{M^n}}
   \sup_{z\in\R^d}\sup_{r,s\in[0,T]}\left\|(\funcF(0)) (r,z+W_s^0) \right\|_{L^2(\P;\R)}
   \tfrac{T^{k+1}}{(k+1)!}
   \\&\quad
   +L\sum_{l=0}^{n-1}
     \left(
       \tfrac{\1_{(0,n)}(l)}{\sqrt{M^{n-l}}}
       +
       \tfrac{\1_{(-\infty,n-1)}(l)}{\sqrt{M^{n-l-1}}}
     \right)
   \left\|
     U_{l,M,Q}^{0}-u^{\infty}
   \right\|_{k+1,Q}.
  \end{split}     \end{equation}
  Next we analyze the \emph{time discretization error}.
Item~\eqref{item:approximations.integrable.ii} of Lemma~\ref{l:approximations.integrable}
and
Item~\eqref{item:approximations.integrable.ii} of Lemma~\ref{l:nonlinear.FK.formula}
ensure
that
  for all $n\in\N$, $s\in[0,T]$, $z\in\R^d$
   it holds $\P$-a.s.\ that
  \begin{equation}  \begin{split}\label{eq:discreteFeynmanKac2}
    \E\!\left[U_{n,M,Q}^{0}(s,z))\right]-u^{\infty}(s,z)
    =
    \E\!\left[\sum_{t\in[s,T]}q^{Q,[s,T]}(t)
    \big(\funcF( U_{n-1,M,Q}^{0})\big)(t,z+W_{t-s}^0)
    -
      \int_s^{T}(\funcF(u^{\infty}))(t,z+W_{t-s}^0)\,dt
    \right].
  \end{split}     \end{equation}
  This,
  the triangle inequality, 
  Lemma~\ref{l:NormMonotonicity},
  Lemma~\ref{l:NormOfInt}, 
  Lemma~\ref{l:FcircNorm2},
  and
  Lemma~\ref{l:NormOfConstants}
 demonstrate for all $n\in\N$, $k\in\N_0\cap[0,2Q-1]$ that
  \begin{align}\label{eq:final.Timeerror}
    &\left\|\E\!\left[U_{n,M,Q}^{0}\right]-u^{\infty}\right\|_{k,Q}
    \\&\leq
    \left\|
    [0,T]\times\R^d\ni(s,z)\mapsto
    \E\!\left[\sum_{t\in[s,T]}q^{Q,[s,T]}(t)
    \big(\funcF( U_{n-1,M,Q}^{0})-\funcF( u^{\infty})\big)(t,z+W_{t-s}^0)
    \right]
    \right\|_{k,Q}
    \nonumber
    \\&
    \nonumber
    +\left\|
      [0,T]\times\R^d\ni(s,z)\mapsto
       \E\!\left[\sum_{t\in[s,T]}q^{Q,[s,T]}(t)
      \big(\funcF( u^{\infty})\big)(t,z+W_{t-s}^0)
      -
      \int_s^{T}(\funcF(u^{\infty}))(t,z+W_{t-s}^0)\,dt
    \right]
    \right\|_{k,Q}
    \\&\leq
     \left\|
    [0,T]\times\R^d\ni(s,z)\mapsto
    \sum_{t\in[s,T]}q^{Q,[s,T]}(t)
    \big(\funcF( U_{n-1,M,Q}^{0})-\funcF( u^{\infty})\big)(t,z+W_{t-s}^0)
    \right\|_{k,Q}+\eps\|1\|_{k,Q}
    \\
    &
    \leq
    \left\|
    \funcF( U_{n-1,M,Q}^{0})-\funcF( u^{\infty})
    \right\|_{k+1,Q}
    \nonumber
    +\eps\|1\|_{k,Q}
    \\&\leq
    \nonumber
    L\left\|
     U_{n-1,M,Q}^{0}- u^{\infty}
    \right\|_{k+1,Q}
    \nonumber
    +\eps\tfrac{T^k}{k!}.
  \end{align}
  In the next step we combine the established bounds for the Monte Carlo error and the time discretization error
  to obtain a bound for the
  \emph{global error}. More formally, observe that 
  \eqref{eq:MC.error} and \eqref{eq:final.Timeerror} ensure that 
  for all $n\in\N$, $k\in\N_0\cap [0,2Q-1]$ it holds 
  that
  \begin{equation}  \begin{split}\label{eq:global.estimate}
    &\left\|U_{n,M,Q}^{0}-u^{\infty}\right\|_{k,Q}\leq \left\|U_{n,M,Q}^{0}-\E\!\left[U_{n,M,Q}^{0}\right]\right\|_{k,Q}
    +\left\|\E\!\left[U_{n,M,Q}^{0}\right]-u^{\infty}\right\|_{k,Q}
  \\&\leq
     \tfrac{1}{\sqrt{M^n}}
    \tfrac{ T^{k}}{k!}
    \Big[
    \sup_{z\in\R^d}
     \sup_{s\in[0,T]}\left\|g(z+W_s^{0})\right\|_{L^2(\P;\R)}
     +T
    \sup_{z\in\R^d}
     \sup_{r,u\in[0,T]}\left\|(\funcF(0)) (r,z+W_u^0) \right\|_{L^2(\P;\R)}
     \Big]
    \\&\qquad
     +\LipConst\sum_{l=0}^{n-1}
     \left(
       \tfrac{\1_{(0,n)}(l)}{\sqrt{M^{n-l-1}}}
       +
       \tfrac{\1_{(-\infty,n-1)}(l)}{\sqrt{M^{n-l-1}}}
     \right)
    \left\|U_{l,M,Q}^{0}-u^{\infty}\right\|_{k+1,Q}
  +
  \LipConst
    \left\|
    U_{n-1,M,Q}^{0}- u^{\infty}
      \right\|_{k+1,Q}
      +\eps \tfrac{T^{k}}{k!}
  \\&=
     \tfrac{1}{\sqrt{M^n}}
    \tfrac{ T^{k}}{k!}
    \Big[
    \sup_{z\in\R^d}
     \sup_{s\in[0,T]}\left\|g(z+W_s^{0})\right\|_{L^2(\P;\R)}
     +T
    \sup_{z\in\R^d}
     \sup_{r,u\in[0,T]}\left\|(\funcF(0)) (r,z+W_u^0) \right\|_{L^2(\P;\R)}
     \Big]
     \\&\qquad
     +\LipConst \tfrac{ \left\|u^{\infty}\right\|_{k+1,Q}}{\sqrt{M^{n-1}}}
     +2\LipConst\sum_{l=1}^{n-1}
       \tfrac{1}{\sqrt{M^{n-l-1}}}
    \left\|U_{l,M,Q}^{0}-u^{\infty}\right\|_{k+1,Q}
      +\eps \tfrac{T^{k}}{k!}.
  \end{split}     \end{equation}
 Hence, we obtain that
  for all $j\in \N_0$, $n\in \N$, $k\in \N_0\cap [0,2Q-1]$ it holds that
  \begin{equation}  \begin{split}
    &\sqrt{M^{-j}}
    \left\|U_{n,M,Q}^{0}-u^{\infty}\right\|_{k,Q}
  \leq
     \tfrac{\LipConst  \left\|u^{\infty}\right\|_{k+1,Q}}{\sqrt{M^{n+j-1}}}
      +\eps \tfrac{T^{k}}{k!\sqrt{M^{j}}}
     +2\LipConst\sum_{l=1}^{n-1}
       \sqrt{M^{-j-n+l+1}}
    \left\|U_{l,M,Q}^{0}-u^{\infty}\right\|_{k+1,Q}
    \\&\quad
    +\tfrac{ T^{k}}{k!\sqrt{M^{j+n}}}
     \Big[
    \sup_{z\in\R^d}
     \sup_{s\in[0,T]}\left\|g(z+W_s^{0})\right\|_{L^2(\P;\R)}
     +T
    \sup_{z\in\R^d}
     \sup_{r,u\in[0,T]}\left\|(\funcF(0)) (r,z+W_u^0) \right\|_{L^2(\P;\R)}
     \Big].
 \end{split}     \end{equation}
 This shows for
 all $n\in\{1,2,\ldots,N\}$ that
 \begin{align}
   &e_n
    \leq
     \LipConst \sup_{k\in\{0,1,\ldots,N-1\}}
      \tfrac{ \left\|u^{\infty}\right\|_{k+1,Q}}{\sqrt{M^{N-k-1}}}
     +\eps e^{T}
     +2\LipConst
     \sum_{l=1}^{n-1}e_l
   \\&\quad
     +\Big[\sup_{i\in\{0,1,\ldots,N-1\}}
      \tfrac{ T^{i}}{i!\sqrt{M^{N-i}}}\Big]
     \Big[
    \sup_{z\in\R^d}
     \sup_{s\in[0,T]}\left\|g(z+W_s^{0})\right\|_{L^2(\P;\R)}
     +T\sup_{z\in\R^d}
    \sup_{r,u\in[0,T]}\left\|(\funcF(0)) (r,z+W_u^0) \right\|_{L^2(\P;\R)}
     \Big].
  \nonumber
 \end{align}
  Combining this with the discrete Gronwall-type inequality in Agarwal \cite[Corollary 4.1.2]{agarwal2000difference}  
  proves that
  \begin{equation}  \begin{split}
    &\left\|U_{N,M,Q}^{0}-u^{\infty}\right\|_{0,Q}
    =
    e_N
    \leq
    \left(1+2\LipConst\right)^{N-1}
    \bigg\{
     \LipConst \sup_{i\in\{1,2,\ldots,N\}}
      \tfrac{ \left\|u^{\infty}\right\|_{i,Q}}{\sqrt{M^{N-i}}}
     +\eps e^{T}
   \\&
     +\Big[\sup_{i\in\{0,1,\ldots,N-1\}}
      \tfrac{ T^{i}}{i!\sqrt{M^{N-i}}}\Big]
     \Big[
      \sup_{z\in\R^d}
     \sup_{s\in[0,T]}\left\|g(z+W_s^{0})\right\|_{L^2(\P;\R)}
     +T
      \sup_{z\in\R^d}
     \sup_{r,u\in[0,T]}\left\|(\funcF(0)) (r,z+W_u^0) \right\|_{L^2(\P;\R)}
     \Big]
     \bigg\}.
  \end{split}     \end{equation}
  This completes the proof of Theorem~\ref{thm:rate}.
\end{proof}
In the proof of the following result, Corollary~\ref{c:rate}, an upper bound 
for the quadrature error on the right-hand side of~\eqref{eq:rate}
is derived under the hypothesis that 
the solution of the PDE is sufficiently smooth and regular.
\begin{corollary}\label{c:rate}
 Assume the setting in Subsection~\ref{sec:setting.full.discretization},
 assume that $u^{\infty}\in C^{\infty}([0,T]\times\R^d,\R)$,
 assume for all $k\in\N_0$, $x\in\R^d$, $t\in[0,T]$ that
 \begin{equation}\label{eq:ldeltaintegrandk}
   \E\!\left[\sup_{s\in[t,T]}
   \left|\left((\tfrac{\partial}{\partial r}+\tfrac{1}{2}\Delta_y)^k u^{\infty}\right)\!(s,x+W_{s-t}^0)\right|
 \right]<\infty,
 \end{equation}
 and let $M,Q\in \N$, $N\in \N \cap [1,2Q)$.
 Then it holds for all $\theta\in\Theta$ that
 \begin{equation}  \begin{split}\label{eq:c.rate}
    &\left\|U_{N,M,Q}^{\theta}-u^{\infty}\right\|_{0,Q}
    \leq
    \left(1+2\LipConst\right)^{N-1}
    \Bigg\{
     \LipConst\sup_{i\in\{1,2,\ldots,N\}}
     \tfrac{\|u^{\infty}\|_{i,Q}}{\sqrt{M^{N-i}}}
   \\&
     +\left[\sup_{i\in\{0,1,\ldots,N-1\}}
      \tfrac{ T^{i}}{i!\sqrt{M^{N-i}}}\right]
      \left[
      \sup_{z\in\R^d}
     \sup_{s\in[0,T]}\left\|g(z+W_s^{0})\right\|_{L^2(\P;\R)}
     +T
      \sup_{z\in\R^d}
     \sup_{r,u\in[0,T]}\Big\|(\funcF(0)) (r,z+W_u^0) \Big\|_{L^2(\P;\R)}
     \right]
   \\&
   +e^T
    \sup_{t\in[0,T]}\Bigg[
    \sup_{u\in[0,t]}\sup_{z\in\R^d}
  \Bigg\|
   \sup_{s\in[t,T]}
    \left|
     \E\!\left[
    \left((\tfrac{\partial}{\partial r}+\tfrac{1}{2}\Delta_y)^{2Q+1}u^{\infty}\right)\!(s,x+W_{s-t}^0)
    \right]
    \right|
 \Big|_{x=z+W_u^0}
   \Bigg\|_{L^2(\P;\R)}
    \tfrac{[Q!]^4(T-t)^{2Q+1}}{(2Q+1)[(2Q)!]^3}
    \Bigg]
    \Bigg\}.
 \end{split}     \end{equation}
\end{corollary}
\begin{proof}
\sloppy{
Throughout this proof assume w.l.o.g.\ that
 $\sup_{z\in\R^d}\sup_{t,s\in[0,T]}
  \E\left[|g(z+W_{t}^0)|
         + |(\funcF(0))(t,z+W_{s}^0)|
         \right]<\infty$ (otherwise the right-hand side of \eqref{eq:c.rate} is infinite and the proof of \eqref{eq:c.rate} is clear).}
Observe that \eqref{eq:ldeltaintegrandk} and the dominated convergence theorem ensure that for every 
$k\in\N_0$, $x\in\R^d$, $t\in[0,T]$ it holds that
 the function
 \begin{equation}  \begin{split}\label{eq:integrandk}
   [t,T]\ni s\mapsto
 \E\!\left[\left((\tfrac{\partial}{\partial r}+\tfrac{1}{2}\Delta_y)^k u^{\infty}\right)\!(s,x+W_{s-t}^0)
 \right]\in\R
 \end{split}     \end{equation}
 is continuous.
  The assumption that $u^{\infty}\in C^{\infty}([0,T]\times\R^d,\R)$
  and
  It\^o's formula imply that 
  for all $x\in\R^d$,
  $t\in[0,T]$, $s\in[t,T]$, $k\in\N$ it holds $\P$-a.s.\ that
  \begin{align}\label{eq:cor.after.ito}
      &\left((\tfrac{\partial}{\partial r}
            +\tfrac{1}{2}\Delta_y
            )^k
      u^{\infty}\right)\!(s,x+W_{s}^{0}-W_t^0)
      -
      \left((\tfrac{\partial}{\partial r}
            +\tfrac{1}{2}\Delta_y
            )^k
       u^{\infty}\right)\!(t,x)
    \\&
    =\int_t^s 
    \left((\tfrac{\partial}{\partial r}+\tfrac{1}{2}\Delta_y)^{k+1} u^{\infty}\right)\!(v,x+W_{v}^0-W_t^0)\,dv
    +\int_t^s 
    \left\langle
    \left(\nabla_y(\tfrac{\partial}{\partial r}+\tfrac{1}{2}\Delta_y)^{k}u^{\infty}\right)\!(v,x+W_{v}^0-W_t^0),
    \,dW_v^0
    \right\rangle.\nonumber
  \end{align}
  \sloppy{This and \eqref{eq:ldeltaintegrandk} show that
 for all $x\in\R^d$, $t\in[0,T]$, $k\in\N$ it holds that
  $\E\big[\sup_{s\in[t,T]}\big|
    \int_t^s 
    \left\langle
    \left(\nabla_y(\tfrac{\partial}{\partial r}+\tfrac{1}{2}\Delta_y)^{k}u^{\infty}\right)(v,x+W_{v}^0-W_t^0),
    \,dW_v^0
    \right\rangle
   \big|\big]<\infty$.
 This implies that
 for all $x\in\R^d$, $t\in[0,T]$, $s\in[t,T]$, $k\in\N$ it holds that
  $\E\big[
    \int_t^s 
    \left\langle
    \left(\nabla_y(\tfrac{\partial}{\partial r}+\tfrac{1}{2}\Delta_y)^{k}u^{\infty}\right)(v,x+W_{v}^0-W_t^0),
    \,dW_v^0
    \right\rangle
   \big]=0$.}
   This, \eqref{eq:cor.after.ito}, and Fubini's theorem show 
   that
  for all $x\in\R^d$,
  $t\in[0,T]$, $s\in[t,T]$, $k\in\N$ it holds that
  \begin{equation}  \begin{split}\label{eq:exp.cor.after.ito}
    &
     \E\!\left[
      \left((\tfrac{\partial}{\partial r}
            +\tfrac{1}{2}\Delta_y
            )^k
      u^{\infty}\right)\!(s,x+W_{s}^{0}-W_t^0)
    \right]
      -
      \left((\tfrac{\partial}{\partial r}
            +\tfrac{1}{2}\Delta_y
            )^k
      u^{\infty}\right)\!(t,x)
    \\&
    =\int_t^s 
     \E\!\left[
    \left((\tfrac{\partial}{\partial r}+\tfrac{1}{2}\Delta_y)^{k+1}u^{\infty}\right)\!(v,x+W_{v}^0-W_t^0)
    \right]dv.
  \end{split}     \end{equation}
  Equation~\eqref{eq:exp.cor.after.ito} (with $k=1$)
  together with
  \eqref{eq:integrandk} (with $k=2$)
   implies 
  for every $x\in\R^d$, $t\in[0,T)$
  that the function
   $[t,T]\ni s\mapsto
   \E\!\left[\left((\tfrac{\partial}{\partial r}+\tfrac{1}{2}\Delta_y)u^{\infty}\right)(s,x+W_{s}^0-W_t^0)
 \right]\in\R$
 is continuously differentiable.
 Induction,~\eqref{eq:integrandk}, and
 \eqref{eq:exp.cor.after.ito}
  prove that
  for every $x\in\R^d$, $t\in[0,T]$
  it holds that the function 
  $[t,T]\ni s\mapsto
   \E\!\left[\left((\tfrac{\partial}{\partial r}+\tfrac{1}{2}\Delta_y)u^{\infty}\right)(s,x+W_{s}^0-W_t^0)
 \right]\in\R$ is infinitely often differentiable.
  This, 
 induction, and~\eqref{eq:exp.cor.after.ito}
 demonstrate that for all $k\in\N$, $x\in\R^d$, $t\in[0,T)$, $s\in[t,T]$
 it holds that
 \begin{equation}  \begin{split}\label{eq:estimate.der}
    &\tfrac{\partial^{k}}{\partial s^{k}}
      \E\!\left[
      \left((\tfrac{\partial}{\partial r}+\tfrac{1}{2}\Delta_y)u^{\infty}\right)\!(s,x+W_{s}^0-W_t^0)
      \right]
    =
     \E\!\left[
    \left((\tfrac{\partial}{\partial r}+\tfrac{1}{2}\Delta_y)^{k+1}u^{\infty}\right)\!(s,x+W_{s}^0-W_t^0)
    \right]
    .
 \end{split}     \end{equation}
 Equation~\eqref{eq:PDE}
 and
 the error representation for the Gau\ss-Legendre quadrature rule
 (see, e.g., \cite[Display (2.7.12)]{davis2007methods})
 imply
 for all $x\in\R^d$, $t\in[0,T)$
 that
 there exists a real number $\xi\in[t,T]$ such that
 \begin{align}
      &
    \sum_{s\in[t,T]}q^{Q,[t,T]}(s)
    \E\!\left[
       \left(\funcF( u^{\infty})\right)\!(s,x+W_{s}^0-W_t^0)
\right]
      -
      \int_t^{T}
      \E\!\left[
      (\funcF(u^{\infty}))(s,x+W_{s}^0-W_t^0)
 \right]
      \,ds
   \\&
   =
     \int_t^{T}
      \E\!\left[
    \left((\tfrac{\partial}{\partial r}+\tfrac{1}{2}\Delta_y)u^{\infty}\right)\!(s,x+W_{s}^0-W_t^0)
 \right]
      \,ds
      -
    \sum_{s\in[t,T]}q^{Q,[t,T]}(s)
    \E\!\left[
    \left((\tfrac{\partial}{\partial r}+\tfrac{1}{2}\Delta_y)u^{\infty}\right)\!(s,x+W_{s}^0-W_t^0)
\right]
    \nonumber
\\&
    =\left(\tfrac{\partial^{2Q}}{\partial s^{2Q}}
      \E\left[
    \left((\tfrac{\partial}{\partial r}+\tfrac{1}{2}\Delta_y)u^{\infty}\right)\!(s,x+W_{s}^0-W_t^0)
 \right]
    \right)\Big|_{s=\xi}
    \tfrac{[Q!]^4(T-t)^{2Q+1}}{(2Q+1)[(2Q)!]^3}.
    \nonumber
 \end{align}
 This and~\eqref{eq:estimate.der}
 prove
 that
 \begin{equation}  \begin{split}
    &
    \sup_{\substack{t\in[0,T],\\ z\in\R^d}}\sup_{u\in[0,t]}
    \bigg\|
    \E\bigg[
    \sum_{s\in[t,T]}q^{Q,[t,T]}(s)
    (\funcF( u^{\infty}))(s,z+W_{u+s-t}^0)
      -
      \int_t^{T}
      (\funcF(u^{\infty}))(s,z+W_{u+s-t}^0)
      \,ds
   \, \Big | \,W_u^{0}
    \bigg]
      \bigg\|_{L^2(\P;\R)}
  \\&
  \leq
    \sup_{t\in[0,T)}\sup_{u\in[0,t]}
    \sup_{z\in\R^d}
    \left\{
  \left\|
    \sup_{s\in[t,T]}\left(\tfrac{\partial^{2Q}}{\partial s^{2Q}}
      \E\!\left[
    \left((\tfrac{\partial}{\partial r}+\tfrac{1}{2}\Delta_y)u^{\infty}\right)\!(s,x+W_{s}^0-W_t^0)
 \right]
 \Big|_{x=z+W_u^0}
    \right)
   \right\|_{L^2(\P;\R)}
    \tfrac{[Q!]^4(T-t)^{2Q+1}}{(2Q+1)[(2Q)!]^3}
    \right\}
 \\&
 \leq
    \sup_{t\in[0,T]}\sup_{u\in[0,t]}
    \sup_{z\in\R^d}
    \left\{
  \left\|
   \sup_{s\in[t,T]}
    \left|
     \E\!\left[
    \left((\tfrac{\partial}{\partial r}+\tfrac{1}{2}\Delta_y)^{2Q+1}u^{\infty}\right)\!(s,x+W_{s}^0-W_t^0)
    \right]
    \right|
 \Big|_{x=z+W_u^0}
   \right\|_{L^2(\P;\R)}
    \tfrac{[Q!]^4(T-t)^{2Q+1}}{(2Q+1)[(2Q)!]^3}
    \right\}.
  \label{eq:estimate.quad.error}
 \end{split}     \end{equation}
 Theorem~\ref{thm:rate} together with~\eqref{eq:estimate.quad.error}
 implies~\eqref{eq:c.rate}.
  The proof of Corollary~\ref{c:rate} is thus completed.
\end{proof}

The following result, Corollary~\ref{c:rate2},
establishes an upper bound for the $L^2$-error between the solution of the PDE
and our approximations~\eqref{eq:def:U} if the $\sup$-norm of the $n$-th derivative of the solution of the PDE
grows sufficiently slowly as $\N\ni n\to\infty$.
\begin{corollary}\label{c:rate2}
 Assume the setting in Subsection~\ref{sec:setting.full.discretization}, 
 assume that $u^{\infty}\in C^{\infty}([0,T]\times\R^d,\R)$,
 let $\alpha \in [0,\nicefrac{1}{4}]$, and let $C\in [0,\infty]$ be the extended
 real number given by
\begin{equation}
\begin{split}
C&= 
    \LipConst\left[\sup_{(t,x)\in[0,T]\times\R^d}|u^{\infty}(t,x)|\right]
    +
     \left[\sup_{x\in\R^d}\left|g(x)\right|\right]
     +T\left[\sup_{(t,x)\in[0,T]\times\R^d}\left|(\funcF(0)) (t,x) \right|\right]
     \\&\quad +
    Te^{T}
   \left[\sup_{k\in\N}\sup_{(t,x)\in[0,T]\times\R^d}
   (k!)^{\alpha-1}
     \left|\left(\!(\tfrac{\partial}{\partial r}+\tfrac{1}{2}\Delta_y)^k
     u^{\infty}\right)\!(t,x)\right|\right].
    \end{split}
\end{equation}
 Then it holds for all $M,Q\in\N$, $N\in \N \cap [0,2Q)$ that
 \begin{equation} \label{eq:c.rate2}
    \sup_{(t,x)\in[0,T]\times\R^d}\left\|U_{N,M,Q}^{0}(t,x)-u^{\infty}(t,x)\right\|_{L^2(\P;\R)}
    \leq C (1+2L)^N\max\!\left\{\tfrac{T^{2Q}}{Q^{2\alpha Q}},\tfrac{\exp(T\sqrt{M})}{M^{N/2}}\right\} .   
 \end{equation}
\end{corollary}
\begin{proof}
To prove \eqref{eq:c.rate2} we assume w.l.o.g.\ that $C\in [0,\infty)$. 
Observe that the Stirling-type formula in Robbins \cite[Displays (1)--(2)]{robbins1955remark} proves for all
$n\in \N$ that
\begin{equation}
\sqrt{2\pi n}\left[\frac{n}{e}\right]^n\le n! \le \sqrt{2\pi n}\left[\frac{n}{e}\right]^ne^{\frac{1}{12}}
\end{equation}
This together with the fact that $\sqrt{e}\le 2$ and the fact that $\forall \, n\in \N\colon \pi e^{\frac{1}{3}}n\le 8^{n}$ shows for all $n\in \N$ that
\begin{equation}\label{eq:stirling}
\begin{split}
\tfrac{n^{2\alpha n}((2n+1)!)^{1-\alpha}[n!]^4}{(2n+1)[(2n)!]^3}
&\le \tfrac{n^{2\alpha n}[n!]^4}{[(2n)!]^{2+\alpha}}
\le \tfrac{n^{2\alpha n}\left[\sqrt{2\pi}n^{n+\frac{1}{2}}e^{-n+\frac{1}{12}}\right]^4}
{\left[\sqrt{2\pi}(2n)^{2n+\frac{1}{2}}e^{-2n}\right]^{2+\alpha}}
=(\sqrt{2\pi})^{2-\alpha}n^{1-\frac{\alpha}{2}}e^{\frac{1}{3}+2n\alpha}2^{-(2n+\frac{1}{2})(2+\alpha)}\\
&\le 2\pi n e^{\frac{1}{3}+\frac{n}{2}}2^{-4n-1}
=\pi e^{\frac{1}{3}}n(\sqrt{e})^n2^{-4n}
\le \pi e^{\frac{1}{3}}n2^{-3n}
\le 1.
\end{split}
\end{equation}
    Next note that Lemma~\ref{l:quad.rule.rn} and~\eqref{eq:seminorms}
  imply that for all $Q\in\N$ , $i\in\{0,1,\ldots,2Q-1\}$ it holds that
  \begin{equation}  \begin{split}\label{eq:norm.uinfty}
    \|u^{\infty}\|_{i,Q}\leq \left[\sup_{(t,x)\in[0,T]\times\R^d}|u^{\infty}(t,x)|\right]
    \left[\sum_{s\in[0,T]}\bar{q}^{i,Q}(s)\right]
    =
    \left[\sup_{(t,x)\in[0,T]\times\R^d}|u^{\infty}(t,x)|\right]
    \frac{T^{i}}{i!}.
  \end{split}     \end{equation}
The assumption that $C\in [0,\infty)$ allows us to apply Corollary~\ref{c:rate} to obtain
  for all $M,Q\in\N$, $N\in \N\cap [0,2Q)$ that
 \begin{equation}  \begin{split}
    &
    \sup_{t\in[0,T]}\sup_{z\in\R^d}\left\|U_{N,M,Q}^{0}(t,z)-u^{\infty}(t,z)\right\|_{L^2(\P;\R)}
    \\&
    \leq
    \sup_{t\in[0,T]}\sup_{z\in\R^d}\sup_{u\in[0,t]}
    \left\|U_{N,M,Q}^{0}(t,z+W_u^0)-u^{\infty}(t,z+W_u^0)\right\|_{L^2(\P;\R)}
    =\left\|U_{N,M,Q}^{0}-u^{\infty}\right\|_{0,Q}
    \\&
    \leq
    \left(1+2\LipConst\right)^{N-1}
    \Bigg\{
    \LipConst\sup_{(t,x)\in[0,T]\times\R^d}|u^{\infty}(t,x)|\sup_{i\in\{0,1,\ldots,N\}}
     \tfrac{T^{i}}{i!\sqrt{M^{N-i}}}
   \\&\quad
     +\sup_{i\in\{0,1,\ldots,N\}}
      \tfrac{ T^{i}}{i!\sqrt{M^{N-i}}}
      \sup_{z\in\R^d}
     \Big[
     \sup_{s\in[0,T]}\left\|g(z+W_s^{0})\right\|_{L^2(\P;\R)}
     +T\sup_{r,u\in[0,T]}\left\|(\funcF(0)) (r,z+W_u^0) \right\|_{L^2(\P;\R)}
     \Big]
   \\&\quad
   +e^{T}
    \sup_{t\in[0,T]}\sup_{u\in[0,t]}
    \sup_{z\in\R^d}
  \bigg\|
   \sup_{s\in[t,T]}
    \left|
     \E\!\left[
    \left((\tfrac{\partial}{\partial r}+\tfrac{1}{2}\Delta_y)^{2Q+1}u^{\infty}\right)\!(s,x+W_{s-t}^0)
    \right]
    \right|
 \Big|_{x=z+W_u^0}
   \bigg\|_{L^2(\P;\R)}
    \tfrac{[Q!]^4(T-t)^{2Q+1}}{(2Q+1)[(2Q)!]^3}
    \Bigg\}
    \\&
    \leq
    \left(1+2\LipConst\right)^{N}
    \Bigg\{
  e^{T}T^{2Q+1}
    \sup_{(t,x)\in[0,T]\times\R^d}
     \left|\left((\tfrac{\partial}{\partial r}+\tfrac{1}{2}\Delta_y)^{2Q+1}
     u^{\infty}\right)\!(t,x)\right|
    \tfrac{ [Q!]^4}{(2Q+1)[(2Q)!]^3}
   \\&\quad
   +
     \tfrac{1}{\sqrt{M^N}}
   \sup_{i\in\{0,1,\ldots,N\}}\left(\tfrac{(\sqrt{M}T)^{i}}{i!}\right)\left[
   \LipConst\sup_{(t,x)\in[0,T]\times\R^d}|u^{\infty}(t,x)|
   +\sup_{x\in\R^d}\left|g(x)\right|
     +T\sup_{(t,x)\in[0,T]\times\R^d}\left|(\funcF(0)) (t,x) \right|\right]
    \Bigg\}.
     \end{split}     \end{equation}
     This, \eqref{eq:stirling}, and the fact that 
     $\sup_{i\in\{0,1,\ldots,N\}} \tfrac{( \sqrt{M}T)^{i}}{i!}\leq e^{T\sqrt{M}}$
  imply for all $M,Q\in\N$, $N\in \N \cap[0,2Q)$ that
      \begin{equation}\label{eq:rate2.estimate}  \begin{split}
    &
    \sup_{(t,x)\in[0,T]\times \R^d}\left\|U_{N,M,Q}^{0}(t,x)-u^{\infty}(t,x)\right\|_{L^2(\P;\R)}
    \\&
  \leq
    \tfrac{\left(1+2\LipConst\right)^{N}}{Q^{2\alpha Q}}
   e^{T}T^{2Q+1}
   \left[\sup_{k\in\N}\sup_{(t,x)\in[0,T]\times\R^d}
   (k!)^{\alpha-1}
     \left|\left(\!(\tfrac{\partial}{\partial r}+\tfrac{1}{2}\Delta_y)^k
     u^{\infty}\right)\!(t,x)\right|\right]
     \left[
   \sup_{n\in\N}
    \tfrac{ n^{2\alpha n}((2n+1)!)^{1-\alpha}[n!]^4}{(2n+1)[(2n)!]^3}
    \right]
   \\&\quad
   +\left(\tfrac{1+2\LipConst}{\sqrt M}\right)^{N}
    e^{T\sqrt{M}}
    \left[
    \LipConst\sup_{(t,x)\in[0,T]\times\R^d}|u^{\infty}(t,x)|
     +
     \sup_{x\in\R^d}\left|g(x)\right|
     +T\sup_{(t,x)\in[0,T]\times\R^d}\left|(\funcF(0)) (t,x) \right|
    \right]\\&
    \leq
    \tfrac{\left(1+2\LipConst\right)^{N}}{Q^{2\alpha Q}}
   e^{T}T^{2Q+1}
   \left[\sup_{k\in\N}\sup_{(t,x)\in[0,T]\times\R^d}
   (k!)^{\alpha-1}
     \left|\left(\!(\tfrac{\partial}{\partial r}+\tfrac{1}{2}\Delta_y)^k
     u^{\infty}\right)\!(t,x)\right|\right]
   \\&\quad
   +\left(\tfrac{1+2\LipConst}{\sqrt M}\right)^{N}
    e^{T\sqrt{M}}
    \left[
    \LipConst\sup_{(t,x)\in[0,T]\times\R^d}|u^{\infty}(t,x)|
     +
     \sup_{x\in\R^d}\left|g(x)\right|
     +T\sup_{(t,x)\in[0,T]\times\R^d}\left|(\funcF(0)) (t,x) \right|
    \right].
 \end{split}     \end{equation}
This establishes~\eqref{eq:c.rate2}.
 The proof of Corollary~\ref{c:rate2} is thus completed.
\end{proof}
The next result, Corollary~\ref{c:rate3}, provides
an upper bound for the $L^2$-error 
between the solution of the PDE
and our approximations~\eqref{eq:def:U} if the parameters $N,M,Q\in \N$ satisfy $N=M=Q$.
Corollary~\ref{c:rate3} is a direct consequence of Corollary \ref{c:rate2}. 
\begin{corollary}\label{c:rate3}
 Assume the setting in Subsection~\ref{sec:setting.full.discretization}, 
 assume that $u^{\infty}\in C^{\infty}([0,T]\times\R^d,\R)$,
 let $\alpha \in [0,\nicefrac{1}{4}]$, and let $C\in [0,\infty]$ be the extended
 real number given by
\begin{equation}
\begin{split}
C&= 
    \LipConst\left[\sup_{(t,x)\in[0,T]\times\R^d}|u^{\infty}(t,x)|\right]
    +
     \left[\sup_{x\in\R^d}\left|g(x)\right|\right]
     +T\left[\sup_{(t,x)\in[0,T]\times\R^d}\left|(\funcF(0)) (t,x) \right|\right]
     \\&\quad +
    Te^{T}
   \left[\sup_{k\in\N}\sup_{(t,x)\in[0,T]\times\R^d}
   (k!)^{\alpha-1}
     \left|\left(\!(\tfrac{\partial}{\partial r}+\tfrac{1}{2}\Delta_y)^k
     u^{\infty}\right)\!(t,x)\right|\right].
    \end{split}
\end{equation}
Then it holds for all $N\in\N$ that
 \begin{equation} \label{eq:c.rate3}
    \sup_{(t,x)\in[0,T]\times\R^d}\left\|U_{N,N,N}^{0}(t,x)-u^{\infty}(t,x)\right\|_{L^2(\P;\R)}
    \leq C \left[\frac{(1+2L)e^T}{N^{2\alpha}}\right]^N.
     \end{equation}
\end{corollary}

\subsection{Analysis of the computational complexity and overall rate of convergence}
\label{sec:comp_cost}

In Lemma~\ref{l:costRN}
$ \operatorname{RN}_{n,M,Q} $ 
is the number of realizations of a scalar standard 
normal random variable
required to compute one realization of the 
random variable 
$
  U_{n,M,Q}^{ \theta }( t, x ) \colon \Omega \to \R.
$
In  Lemma~\ref{l:costFE}
$ \operatorname{FE}_{n,M,Q} $
is the number of function evaluations of $ f $ and $ g $
required to compute one realization of
$
  U_{n,M,Q}^{ \theta }( t, x ) \colon \Omega \to \R
$.

\begin{lemma}\label{l:costRN}
 Assume the setting in Subsection~\ref{sec:setting.full.discretization}
 and
 let $(\operatorname{RN}_{n,M,Q})_{n,M,Q\in\Z}\subseteq\N_0$ be natural numbers which satisfy
for all $n,M,Q \in \N$
that $\operatorname{RN}_{0,M,Q}=0$
and 
\begin{align}
\label{eq:RN}
  \operatorname{RN}_{ n,M,Q }
  &\leq d M^n+\sum_{l=0}^{n-1}\left[Q M^{n-l}( d + \operatorname{RN}_{ l, M,Q }+ \mathbbm{1}_{ \N }( l ) \cdot\operatorname{RN}_{ l-1, M,Q })\right].
\end{align}
Then for all $N\in\N$, we have
$$
\operatorname{RN}_{ N,N,N }
\leq 8 d N^{2N}.
$$
\end{lemma}
\begin{proof}
Inequality~\eqref{eq:RN} implies
for all $ n, Q \in \N $, $M \in \N\cap [2,\infty)$
that
\begin{equation}  \begin{split}\label{eq:real_stand_norm}
  (M^{-n}\cdot \operatorname{RN}_{ n,M,Q })
  &\leq d+\sum_{l=0}^{n-1}\left[Q M^{-l}( d + \operatorname{RN}_{ l, M,Q }+ \mathbbm{1}_{ \N }( l ) \cdot\operatorname{RN}_{ l-1, M,Q })\right]\\
  &\leq  d  \left(1+\tfrac{MQ}{M-1}\right)+(1+\tfrac{1}{M})Q\left[\sum_{l=0}^{n-1}(M^{-l}\cdot \operatorname{RN}_{ l, M,Q })\right].
\end{split}     \end{equation}
The fact that $\forall\, M,Q \in \N\colon\operatorname{RN}_{0,M,Q}=0$ and
the discrete Gronwall-type inequality in Agarwal \cite[Corollary 4.1.2]{agarwal2000difference} hence prove
that for all $ n, Q \in \N $, $M \in \N\cap [2,\infty)$ it holds that
\begin{equation}
 (M^{-n}\cdot\operatorname{RN}_{ n,M,Q })
  \leq
  d\left(1+\tfrac{MQ}{M-1}\right)(1+(1+\tfrac{1}{M})Q)^{n-1}
  \leq 
  \tfrac{d(M+(M+1)Q)^{n}}{M^{n-1}(M-1)}.
\end{equation}
Hence, we obtain
that for all $N\in \N\cap [2,\infty)$ it holds that
\begin{equation}
\operatorname{RN}_{ N,N,N }\leq \tfrac{Nd}{N-1}(N+(N+1)N)^N
= \tfrac{N}{N-1}(1+\tfrac{2}{N})^N d N^{2N}
\leq 8 d N^{2N}.
\end{equation}
This and the fact that $\operatorname{RN}_{ 1,1,1 }\leq 2d$ complete
the proof of Lemma~\ref{l:costRN}.
\end{proof}

\begin{lemma}\label{l:costFE}
 Assume the setting in Subsection~\ref{sec:setting.full.discretization}
 and
 let $(\operatorname{FE}_{n,M,Q})_{n,M,Q\in\Z}\subseteq\N_0$ be natural numbers which satisfy
for all $n,M,Q \in \N$
that $\operatorname{FE}_{0,M,Q}=0$
and 
\begin{align}
  \operatorname{FE}_{ n,M,Q }
  &\leq M^n+\sum_{l=0}^{n-1}\left[Q M^{n-l}( 1 + \operatorname{FE}_{ l, M,Q }+ \mathbbm{1}_{ \N }( l )+ \mathbbm{1}_{ \N }( l ) \cdot\operatorname{FE}_{ l-1, M,Q })\right].
  \label{eq:FE}
\end{align}
Then for all $N\in\N$, we have
$$
\operatorname{FE}_{ N,N,N }
\leq  8 N^{2N}.
$$
\end{lemma}
The proof of Lemma~\ref{l:costFE} is analogous to the proof of Lemma~\ref{l:costRN} and therefore omitted.
In the proof of Corollary~\ref{c:rate4a} below we combine Lemma~\ref{l:costRN} and Lemma~\ref{l:costFE} with Corollary~\ref{c:rate3} to obtain a bound for the computational complexity of our scheme \eqref{eq:def:U} in terms of the space dimension and the prescribed approximation accuracy.

The next result, Corollary~\ref{c:rate4a},
proves under suitable assumptions 
that 
if $\eps\in(0,\infty)$ is the prescribed 
approximation accuracy and if $d\in\N$ is the dimension of the considered PDE,
then for every $\alpha \in (0,\nicefrac{1}{4}]$ and every $\delta \in (0,\infty)$ it holds that the computational effort of the approximation method (number of function evaluations of the coefficient functions of the
considered PDE and number of used independent scalar standard normal random variables, cf.\ Section~\ref{sec:comp_cost}) is at most $O(d\,\eps^{-(\frac{1}{\alpha}+\delta)})$.
\begin{corollary}\label{c:rate4a}
 Assume the setting in Subsection~\ref{sec:setting.full.discretization},
 assume that $u^{\infty}\in C^{\infty}([0,T]\times\R^d,\R)$,
 let $\alpha \in (0,\nicefrac{1}{4}]$, $\delta \in (0,\infty)$, 
 let $C\in [0,\infty]$ be the extended
 real number
given by
 \begin{equation}\label{eq:constant_rate4a}
\begin{split}
C&= 16\exp\!\left(2\alpha \delta [e^T(1+2L)]^{\frac{1+\alpha\delta}{2\alpha^2 \delta}}\right)
\Bigg\{ 
    \LipConst\left[\sup_{(t,x)\in[0,T]\times\R^d}|u^{\infty}(t,x)|\right]
    +
     \left[\sup_{x\in\R^d}\left|g(x)\right|\right]
     \\&
     +T\left[\sup_{(t,x)\in[0,T]\times\R^d}\left|(\funcF(0)) (t,x) \right|\right]
      +
    Te^{T}
   \left[\sup_{k\in\N}\sup_{(t,x)\in[0,T]\times\R^d}
   (k!)^{\alpha-1}
     \left|\left(\!(\tfrac{\partial}{\partial r}+\tfrac{1}{2}\Delta_y)^k
     u^{\infty}\right)\!(t,x)\right|\right]\Bigg\}^{\nicefrac{1}{\alpha}+\delta},
    \end{split}
\end{equation}
let $(\operatorname{RN}_{n,M,Q})_{n,M,Q\in\Z}\subseteq\N_0$ be natural numbers
which satisfy
for all $n,M,Q \in \N$
that $\operatorname{RN}_{0,M,Q}=0$
and
\begin{align}
  \operatorname{RN}_{ n,M,Q }
  &\leq d M^n+\sum_{l=0}^{n-1}\left[Q M^{n-l}( d + \operatorname{RN}_{ l, M,Q }+ \mathbbm{1}_{ \N }( l ) \cdot\operatorname{RN}_{ l-1, M,Q })\right]
\end{align}
(for every $N\in\N$ 
 we think of
 $\operatorname{RN}_{ N,N,N }$
 as the number of realizations of a scalar standard normal random variable required
 to compute one realization of the random variable
 $U^{0}_{N,N,N}(0,0)\colon\Omega\to\R$),
and let $( \operatorname{FE}_{n,M,Q})_{n,M,Q\in\Z}\subseteq\N_0$ be natural numbers
which satisfy
for all $n,M,Q \in \N$
that $\operatorname{FE}_{0,M,Q}=$
and
\begin{equation}  \begin{split}
  \operatorname{FE}_{ n,M,Q }
  &\leq M^n+\sum_{l=0}^{n-1}\left[Q M^{n-l}( 1 + \operatorname{FE}_{ l, M,Q }+ \mathbbm{1}_{ \N }( l )+ \mathbbm{1}_{ \N }( l ) \cdot\operatorname{FE}_{ l-1, M,Q })\right]
\end{split}     \end{equation}
(for every $N\in\N$ 
 we think of
 $\operatorname{FE}_{ N,N,N }$
 as the
 number of function evaluations of $f$ and $g$
 required
 to compute one realization of the random variable
 $U^{0}_{N,N,N}(0,0)\colon\Omega\to\R$).
 Then it holds for all $N\in \N$ that
 \begin{equation}  \begin{split}\label{eq:c.rate4a}
 \operatorname{RN}_{N,N,N}+\operatorname{FE}_{N,N,N}
 \le 
 Cd \left[\sup_{(t,x)\in[0,T]\times\R^d}\left\|U_{N,N,N}^{0}(t,x)-u^{\infty}(t,x)\right\|_{L^2(\P;\R)}\right]^{-\left(\nicefrac{1}{\alpha}+\delta\right)}.
 \end{split}     \end{equation}
\end{corollary}
\begin{proof}
We assume w.l.o.g.\ that $C\in [0,\infty)$. Throughout this proof let $\tilde C \in [0,\infty)$ be the real number given by $\tilde C=\frac{1}{16}\exp\!\left(-2\alpha \delta [e^T(1+2L)]^{\frac{1+\alpha\delta}{2\alpha^2 \delta}}\right)C$. Corollary \ref{c:rate3}, 
Lemma~\ref{l:costRN},
and
Lemma~\ref{l:costFE}
 prove that for all $N\in \N$ it holds that
 \begin{multline}
 \left(\operatorname{RN}_{N,N,N}+\operatorname{FE}_{N,N,N}\right)
\left[    
    \sup_{(t,x)\in[0,T]\times\R^d}\left\|U_{N,N,N}^{0}(t,x)-u^{\infty}(t,x)\right\|_{L^2(\P;\R)}
    \right]^{\nicefrac{1}{\alpha}+\delta}
   \\
    \leq  \left(8dN^{2N}+8N^{2N}\right)\tilde C
    \left[\frac{(1+2L)e^T}{N^{2\alpha}}\right]^{N\left(\nicefrac{1}{\alpha}+\delta\right)}
 = 8(d+1) \tilde C [(1+2L)e^T]^{N\left(\nicefrac{1}{\alpha}+\delta\right)}N^{-2\alpha \delta N}.
 \end{multline}
 This and the fact that $\forall\, N\in \N\colon N!\le N^N$ show that for all $N\in \N$ it holds that
\begin{equation}
\begin{split}
 &\left(\operatorname{RN}_{N,N,N}+\operatorname{FE}_{N,N,N}\right)
\left[    
    \sup_{(t,x)\in[0,T]\times\R^d}\left\|U_{N,N,N}^{0}(t,x)-u^{\infty}(t,x)\right\|_{L^2(\P;\R)}
    \right]^{\nicefrac{1}{\alpha}+\delta}
   \\
    &\leq 16d \tilde C \tfrac{[(1+2L)e^T]^{N\left(\nicefrac{1}{\alpha}+\delta\right)}}{(N!)^{2\alpha \delta}}
    = 16d \tilde C \left[\tfrac{[(1+2L)e^T]^{N\left(\frac{1+\alpha\delta}{2\alpha^2\delta}\right)}}{N!}\right]^{2\alpha \delta}
    \leq 16d \tilde C \left[\sum_{n=0}^\infty \tfrac{[(1+2L)e^T]^{n\left(\frac{1+\alpha\delta}{2\alpha^2\delta}\right)}}{n!}\right]^{2\alpha \delta}\\
    &
    = 16d \tilde C \left[\exp\!\left([(1+2L)e^T]^{\frac{1+\alpha\delta}{2\alpha^2\delta}}\right)\right]^{2\alpha \delta}
    = 16d \tilde C \left[\exp\!\left(2\alpha \delta[(1+2L)e^T]^{\frac{1+\alpha\delta}{2\alpha^2\delta}}\right)\right]=Cd.
    \end{split}
 \end{equation}
This completes the proof of Corollary~\ref{c:rate4a}.
\end{proof}

The next result, Corollary~\ref{c:rate4},
specializes Corollary~\ref{c:rate4a} to the case
$\alpha=\nicefrac{1}{4}$.
\begin{corollary}\label{c:rate4}
 Assume the setting in Subsection~\ref{sec:setting.full.discretization},
 assume that $u^{\infty}\in C^{\infty}([0,T]\times\R^d,\R)$,
 let $\delta \in (0,\infty)$, 
 let $C\in [0,\infty]$ be the extended
 real number
given by
 \begin{equation}\label{eq:constant_rate4}
\begin{split}
C&= 16\exp\!\left( \delta [e^T(1+2L)]^{2+(\nicefrac{8}{\delta})}\right)
\Bigg\{ 
    \LipConst\left[\sup_{(t,x)\in[0,T]\times\R^d}|u^{\infty}(t,x)|\right]
    +
     \left[\sup_{x\in\R^d}\left|g(x)\right|\right]
     \\&\quad
     +T\left[\sup_{(t,x)\in[0,T]\times\R^d}\left|(\funcF(0)) (t,x) \right|\right]
      +
    Te^{T}
   \left[\sup_{k\in\N}\sup_{(t,x)\in[0,T]\times\R^d}
  \frac{\left|\left(\!(\tfrac{\partial}{\partial r}+\tfrac{1}{2}\Delta_y)^k
     u^{\infty}\right)\!(t,x)\right|} {(k!)^{\nicefrac{3}{4}}}
     \right]\Bigg\}^{4+\delta},
    \end{split}
\end{equation}
let $(\operatorname{RN}_{n,M,Q})_{n,M,Q\in\Z}\subseteq\N_0$ be natural numbers
which satisfy
for all $n,M,Q \in \N$
that $\operatorname{RN}_{0,M,Q}=0$
and
\begin{align}
  \operatorname{RN}_{ n,M,Q }
  &\leq d M^n+\sum_{l=0}^{n-1}\left[Q M^{n-l}( d + \operatorname{RN}_{ l, M,Q }+ \mathbbm{1}_{ \N }( l ) \cdot\operatorname{RN}_{ l-1, M,Q })\right]
\end{align}
(for every $N\in\N$ 
 we think of
 $\operatorname{RN}_{ N,N,N }$
 as the number of realizations of a scalar standard normal random variable required
 to compute one realization of the random variable
 $U^{0}_{N,N,N}(0,0)\colon\Omega\to\R$),
and let $( \operatorname{FE}_{n,M,Q})_{n,M,Q\in\Z}\subseteq\N_0$ be natural numbers
which satisfy
for all $n,M,Q \in \N$
that $\operatorname{FE}_{0,M,Q}=$
and
\begin{equation}  \begin{split}
  \operatorname{FE}_{ n,M,Q }
  &\leq M^n+\sum_{l=0}^{n-1}\left[Q M^{n-l}( 1 + \operatorname{FE}_{ l, M,Q }+ \mathbbm{1}_{ \N }( l )+ \mathbbm{1}_{ \N }( l ) \cdot\operatorname{FE}_{ l-1, M,Q })\right]
\end{split}     \end{equation}
(for every $N\in\N$ 
 we think of
 $\operatorname{FE}_{ N,N,N }$
 as the
 number of function evaluations of $f$ and $g$
 required
 to compute one realization of the random variable
 $U^{0}_{N,N,N}(0,0)\colon\Omega\to\R$).
 Then it holds for all $N\in \N$ that
 \begin{equation}  \begin{split}\label{eq:c.rate4}
 \operatorname{RN}_{N,N,N}+\operatorname{FE}_{N,N,N}
 \le 
 Cd \left[\sup_{(t,x)\in[0,T]\times\R^d}\left\|U_{N,N,N}^{0}(t,x)-u^{\infty}(t,x)\right\|_{L^2(\P;\R)}\right]^{-\left(4+\delta\right)}.
 \end{split}     \end{equation}
\end{corollary}

\subsubsection*{Acknowledgement}
This project has been partially supported through the research grants ONR N00014-13-1-0338, DOE DE-SC0009248, and
by the Deutsche Forschungsgesellschaft (DFG) via research grant HU 1889/6-1.

%
%
%

{\small
\bibliographystyle{acm}
\bibliography{../Bib/bibfile}
}

\end{document}